\documentclass[12pt]{amsart}
 \pdfoutput=1
\usepackage{microtype}
\usepackage{etex} 
\usepackage[active]{srcltx}
\usepackage{calc,amssymb,amsthm,amsmath,amscd, eucal,ulem}
\usepackage{alltt}
\usepackage{mathtools}
\usepackage{bbding}
\RequirePackage[dvipsnames,usenames]{color}

\input{mabliautoref.sty}
\input{xy}
\xyoption{all}
\input{kmacros3.sty}
\usepackage{tikz}
\usepackage{amsfonts, mathrsfs}
\usepackage[cal=boondox, calscaled=1.05]{mathalfa}
\usepackage{calligra}
\usepackage{amssymb}
\usepackage{stmaryrd} 
\usepackage{dcpic, pictexwd}
\usepackage[left=1in,top=1in,right=1in,bottom=1in]{geometry}
\usepackage{bm}
\usepackage{verbatim}
\usepackage{upgreek}
\usepackage{systeme}
\usepackage{colonequals}
\usepackage{url}

\numberwithin{equation}{theorem}

\newcommand{\eg}{{\itshape e.g.} }

\renewcommand{\m}{\mathfrak{m}}

\newcommand{\p}{\mathfrak{p}}
\newcommand{\q}{\mathfrak{q}}

\DeclarePairedDelimiter\ceil{\lceil}{\rceil}

\DeclareMathOperator{\RHom}{RHom}

\usepackage{setspace}
\usepackage{hyperref}
\usepackage{enumerate}
\usepackage{enumitem}
\usepackage{graphicx}
\usepackage[all,cmtip]{xy}

\usepackage{verbatim}

\theoremstyle{theorem}

\newtheorem*{mainthma}{Theorem A}
\newtheorem*{mainthmb}{Theorem B}
\newtheorem*{mainthmc}{Theorem C}
\newtheorem*{mainthmd}{Theorem D}

\renewcommand{\sA}{\mathcal{A}}

\renewcommand{\sC}{\mathcal{C}}

\renewcommand{\sH}{\mathcal{H}}
\renewcommand{\sI}{\mathcal{I}}
\renewcommand{\sJ}{\mathcal{J}}

\renewcommand{\sO}{\mathcal{O}}

\renewcommand{\sU}{\mathcal{U}}

\usepackage{marvosym}

\newcommand{\unpe}{{1/{p^e}}}
\newcommand{\unp}{{1/{p}}}
\renewcommand{\pe}{{p^e}}

\newcommand{\fe}{F^e}
\newcommand{\fstare}{\fe_*}

\newcommand{\ps}{\p S}

\newcommand{\rphi}{(R,\phi)}
\newcommand{\spsi}{(S,\psi)}

\newcommand{\pS}{\p S}

\newcommand{\inv}{^{-1}}

\let\Im\relax
\DeclareMathOperator{\Im}{Im}

\newcommand{\unpfty}{{1/p^\infty}}
\DeclareMathOperator{\perfdd}{perfd}
\DeclareMathOperator{\perfd}{\!_{perfd}}
\DeclareMathOperator{\perf}{\!_{perf}}

\newcommand{\fracc}{\mathfrak{c}}
\newcommand{\rainfty}{R^A_\infty}
\newcommand{\rafty}{R^A_\infty}
\newcommand{\rbfty}{R^B_\infty}

\newcommand{\safty}{S^A_\infty}
\newcommand{\ranfty}{R^{N,A}_\infty}
\newcommand{\aafty}{\fra^A_\infty}
\newcommand{\bafty}{\frab^A_\infty}

\newcommand{\cafty}{\fracc^A_\infty}

\newcommand{\rgamma}{\bold{R}\Gamma}

\usepackage{relsize}
\usepackage[bbgreekl]{mathbbol}

\DeclareSymbolFontAlphabet{\mathbb}{AMSb}
\DeclareSymbolFontAlphabet{\mathbbl}{bbold}

\begin{document}
\title{Centers of perfectoid purity}

\author[A.~Fayolle]{Anne Fayolle}
\address{Department of Mathematics\\ University of Utah\\ Salt Lake City\\ UT 84112\\USA}
\email{\href{mailto:fayolle@math.utah.edu}{fayolle@math.utah.edu}}

\keywords{$F$-singularities, centers of $F$-purity, perfectoid singularities, log canonical center}
\subjclass[2020]{14G45, 14F18, 14B05, 13A35}

\begin{abstract}
We introduce a mixed characteristic analog of log canonical centers in characteristic $0$ and centers of $F$-purity in positive characteristic, which we call centers of perfectoid purity.
We show that their existence detects (the failure of) normality of the ring. We also show the existence of a special center of perfectoid purity that detects the perfectoid purity of $R$, analogously to the splitting prime of Aberbach and Enescu, and investigate its behavior under étale morphisms.
\end{abstract}
\maketitle

\section{Introduction}

Let $R$ be a commutative Noetherian ring of characteristic $p>0$. Since the late sixties, the singularities of the ring $R$ have been studied using Frobenius; see \cite{KunzCharacterizationsOfRegularLocalRings,HochsterHunekeTC1,HochsterFoundations,HochsterHunekeTightClosureAndStrongFRegularity,SmithFRatImpliesRat,MehtaRamanathanFrobeniusSplittingAndCohomologyVanishing} to mention a few. 
These are broadly known as $F$-singularities and have been linked to singularities of the Minimal Model Program \cite{SchwedeFInjectiveAreDuBois,SmithFRatImpliesRat,SmithMultiplierTestIdeals}.
Of particular importance in this theory are the \emph{centers of $F$-purity} \cite{SchwedeCentersOfFPurity}, a special type of \emph{compatible ideal} \cite{MehtaRamanathanFrobeniusSplittingAndCohomologyVanishing}. They tell us where the ring fails to be (strongly) $F$-regular.  
These are related to log canonical centers, an important object in the study of singularities of the Minimal Model Program \cite{AmbroLCCENTERS}. 
The aim of this paper is to define an analogous object in mixed characteristic. 

Although there is no Frobenius in mixed characteristic, we have a good analog of perfection: perfectoidization \cite{ScholzePerfectoidSpaces,BhattScholzePrisms}. Our strategy is then to express what centers of $F$-purity are in terms of perfection in positive characteristic before writing an analog definition in mixed characteristic via perfectoidization. This strategy has been used successfully to define analogs of test ideals, $F$-signature, and $F$-purity in mixed characteristic \cite{MaSchwedePerfectoidTestIdeal,CaiLeeMaSchwedeTuckerPerfectoidSignature,BMPSTWWPerfectoidPure}.

Suppose that $R$ is reduced and let $F\colon R\to R^\unp$ be the Frobenius map on $R$. Assume furthermore that $R^\unp$ is a finite $R$-module and that the natural map $R\to R^\unp$ is pure. An ideal $\fra$ of $R$ is said to be \emph{uniformly compatible} if $\phi(\fra^\unpe)\subset \fra$ for all $e\in \N$, $\phi\in \Hom_R(R^\unpe,R)$. Let $R\perf\coloneqq \bigcup_e R^\unpe$ and $\fra\perf\coloneqq \bigcup_e \fra^\unpe$. Then, $\fra$ is uniformly compatible if and only if $\psi(\fra\perf)\subset \fra$ for all $\psi\in \Hom_R(R\perf,R)$
(see \autoref{corollary equiv comp with char p}). 
When $\fra$ is prime, we say that it is a \emph{center of $F$-purity} of $R$.

Now, let $(R,\fram)$ be a complete Noetherian local ring with perfect residue field $k$ of characteristic $p>0$. Suppose that $R$ is a finite $A$-algebra where 
$A\coloneqq W(k)\llbracket x_1,\ldots,x_d \rrbracket $ (for instance a Noether normalization). Denoting by perfd the perfectoidization of a ring or an ideal (see \cite[Section 10]{BhattScholzePrisms} or \autoref{definition perfectoidization}), we define 
\[
\rafty\coloneqq \left(R\tensor_A W(k)\left[p^\unpfty,x_1^\unpfty, \ldots,x_d^\unpfty\right]\llbracket x_1,\ldots,x_d\rrbracket^{\wedge p}\right)\perfd.
\]
Note that if $R$ itself has characteristic $p$ then $\rafty=R\perf$ (\autoref{corollary equiv comp with char p}). Let $\fra\subset R$ be an ideal and fix $\phi\in \Hom_R(\rafty,R)$. We say that $\fra$ is \emph{$\phi$-compatible} if 
\[\phi\left(\left (\fra\rafty\right)\perfd \right) \subset \fra.
\]
If this holds for every possible choice of $\phi\in \Hom_R(\rafty,R)$, we say that $\fra$ is \emph{uniformly perfectoid compatible}. This does not depend on the choice of $A$ by \autoref{uniformly compatible doesn't depend on the choice of A}. Moreover, like in positive characteristic, these are closed under sums, intersections, and minimal primes, see \autoref{lemma compatible ideals closed under sums intersection}, \autoref{compatible ideals of pairs closed under ass primes}. This allows us to prove the following theorem: 

\begin{mainthma}[{\autoref{finiteness of uniformly compatible ideals}}]
Let $(R,\fram)$ be a complete Noetherian local ring with residue field of characteristic $p>0$. If $R$ is perfectoid pure, there are only finitely many uniformly perfectoid compatible ideals. 
\end{mainthma}

In fact, there are finitely many $\phi$-compatible ideals for a surjective $\phi$ (\autoref{finiteness of phi compatible ideals}). If $R$ is perfectoid pure and $\p$ is a uniformly perfectoid compatible prime ideal of $R$, we say that it is a \emph{center of perfectoid purity} of $R$. Not only are there finitely many of them but we actually have a bound thanks to \cite{SchwedeTuckerNumberOfFSplit}, see \autoref{remark bound number cpps}.
In positive characteristic, log canonical centers are centers of $F$-purity \cite{SchwedeCentersOfFPurity}. The same holds true in mixed characteristic when $R$ is quasi-Gorenstein (that is, when $R$ has a canonical module that is free of rank $1$):

\begin{mainthmb}[{\autoref{lc center is cpp for big enough A}, \autoref{multiplier ideal uniformly compatible}}]
Let $(R,\fram)$ be a complete Noetherian normal quasi-Gorenstein local ring with residue field of characteristic $p>0$. If $R$ is perfectoid pure then,
\begin{enumerate}
    \item The multiplier ideal $\sJ$ of $R$ is a uniformly perfectoid compatible ideal of $R$. 
    \item The ideals defining log canonical centers of $R$ are centers of perfectoid purity.
    \item If $R$ has no uniformly perfectoid compatible ideal, then it is klt.
\end{enumerate}
\end{mainthmb}

An important aspect of the theory in positive characteristic is that compatible ideals detect whether the ring is normal: if a ring of characteristic $p>0$ has no uniformly compatible ideal, it must be normal. This is also true in mixed characteristic:
\begin{mainthmc}[{\autoref{conductor ideal is compatible}, \autoref{conductor uniformly compatible}}]
Let $(R,\fram)$ be a complete Noetherian local ring with residue field of characteristic $p>0$. Then, the conductor ideal $\fracc$ of $R$ is a nonzero uniformly perfectoid compatible ideal of $R$. In particular, if $R$ has no uniformly perfectoid compatible ideal, then it is normal. 
\end{mainthmc}

We also show the existence of a special uniformly perfectoid compatible ideal $\upbeta(R)\subset R$ that detects perfectoid purity, analogous to the splitting prime of Aberbach and Enescu \cite{AberbachEnescuStructureOfFPure}. This is defined in \autoref{definition uniform splitting prime} and has the following property. 

\begin{mainthmd}[{\autoref{basic properties uniform splitting prime}}]
Let $(R,\fram)$ be a complete Noetherian local ring with perfect residue field of characteristic $p>0$. Then, $\upbeta(R)\neq R$ if and only $R$ is perfectoid pure. In that case, $\upbeta(R)$ is a prime and is the largest center of perfectoid purity of $R$. In particular, $R/\upbeta(R)$ is normal. 
\end{mainthmd}

We then generalize this construction to create an analog of the Cartier core map \cite{BadillaCespedesFInvariantsSRRings,BrosowskyCartierCoreMap,CarvajalFayolleTameRamificationCFPS} in \autoref{section compatible cores}. We also show that this map behaves well under étale morphisms in \autoref{section behavior étale morphisms}.

\subsection{Acknowledgements}
The research was partially supported by an NSERC doctoral grant. I would like to thank my advisor, Karl Schwede, for his constant support and generous help during the preparation of this paper. His many insights can be found throughout this work. I would like to thank Linquan Ma for his help in improving and correcting a previous version of this paper. Thanks also to the anonymous referee for valuable comments and suggestions.

\section{Definition and first properties}
\begin{notation}
For this whole paper, all rings are assumed to be commutative with unity. If $(R,\fram)$ is a local ring and $M$ is an $R$-module, then $\widehat{M}$ or $M^\wedge$ is the classical $\fram$-adic completion of $M$. If $\fra$ is any other ideal of $R$, $M^{\wedge \fra}$ is the classical $\fra$-adic completion of $M$ unless explicitly stated otherwise. Denoting by $p$ the characteristic of $R/\fram$, if $S$ is an $R$-algebra, and $\fra\subset S$ is an ideal, then $\fra^-$ is the $p$-adic closure of $\fra$ in $S$.
\end{notation}
\subsection{Generalities about perfectoid rings and ideals}
\begin{definition}[{\cite[Section 10]{BhattScholzePrisms}}]\label{definition perfectoidization}
Let $R$ be a perfectoid ring and $\fra\subset R$ an ideal. We say that $\fra$ is a perfectoid ideal if $R/\fra$ is also perfectoid. If $\frab\subset R$ is any ideal, we define $\frab\perfd\coloneqq \ker(R\to (R/\frab)\perfd)$ and $\frab$ is perfectoid if and only if $\frab=\frab\perfd$.
\end{definition}

We write down a couple of facts about perfectoid ideals. These are well known to experts but will be useful throughout the paper.

\begin{proposition}\label{kernel of map of perfd rings}
Let $f\colon R\to S$ be a morphism between perfectoid rings. Then, $\ker f$ is a perfectoid ideal of $R$. 
\end{proposition}
\begin{proof}
As the kernel of a map of two $p$-complete rings, $\ker f$ is also $p$-complete. In particular, $R/\ker f$ is semiperfectoid so it surjects onto its perfectoidization by \cite[Theorem 7.4]{BhattScholzePrisms}. On the other hand, $R/\ker f\into S$ factors through $(R/\ker f)\perfd$ so $R/\ker f\to (R/\ker f)\perfd$ must be injective, hence an isomorphism. 
\end{proof}

\begin{proposition}[{\cite[Proposition 2.8]{DineIshizukaTiltingUntilting}}]\label{perfectoidization is radical when p in ideal}
Let $R$ be a perfectoid ring and $\fra\subset R$ an ideal containing $p$. Then, $\fra$ is perfectoid if and only if it is radical. 
\end{proposition}

\begin{proposition}[{\cite[Example 7.9]{BhattScholzePrisms}, \cite[Lemma 2.4.3]{CaiLeeMaSchwedeTuckerPerfectoidSignature}}]\label{perfectoidization of principal ideal with roots}
Let $R$ be a perfectoid ring and $f\in R$ be such that $f$ has a compatible system of $p$-th power roots, which we denote by $\{f^{1/{p^\infty}}\}$. Then, $(f)\perfd=(f^\unpfty)^{-}$. 
\end{proposition}

\begin{proposition}\label{intersection of perfectoid ideals is perfectoid}
Let $R$ be a perfectoid ring and $\{\fra_i\}_{i\in I}$ be a set of perfectoid ideals. Then, $\cap_{i\in I} \fra_i$ is a perfectoid ideal of $R$. 
Moreover, if $I$ is a finite set and $\{\frab_i\}_{i\in I}$ are $p$-complete ideals of $R$, then $\cap_i (\frab_i)\perfd=(\cap_i \frab_i)\perfd$.
\end{proposition}
\begin{proof}
By \cite[Example 3.8 (8)]{BhattIyengarMaRegularRings}, the product
\[
\prod_{i\in I}R/\fra_i
\]
is perfectoid. Now, 
\[
\bigcap_{i\in I} \fra_i=\ker \left( R\to \prod_{i\in I}R/\fra_i\right)
\]
so the result follows from \autoref{kernel of map of perfd rings}. 
For the second part, since sheafification commutes with finite limits, \cite[Corollary 8.11]{BhattScholzePrisms} gives 
\[
\left(\prod_{i\in I} R/\frab_i\right)\perfd=\prod_{i\in I} \left( R/\frab_i\right)\perfd.
\]
In particular, 
\[
\bigcap_i (\frab_i)\perfd=\ker\left( R\to \prod_{i\in I} \left( R/\frab_i\right)\perfd \right)= \ker \left( R\to \left(\prod_{i\in I} R/\frab_i\right)\perfd\right),
\]
which is precisely $\left(\bigcap_i \frab_i\right)\perfd$, as desired.
\end{proof}

\begin{proposition}\label{finite sum of perfectoid ideals is perfectoid}
Let $R$ be a perfectoid ring and $\fra$ and $\frab$ be two perfectoid ideals of $R$. Then, $\fra+\frab$ is also a perfectoid ideal of $R$.
\end{proposition}
\begin{proof}
We have an exact sequence 
\begin{equation}\label{intersection sum exact sequence}   
0\to R/\left(\fra\cap \frab\right) \to R/\fra \oplus R/\frab \to R/\left(\fra + \frab\right)\to 0,
\end{equation}
which we claim implies $R/\left(\fra + \frab\right)$ is also perfectoid. Indeed,
\[
V\left(\fra\right)\cup V\left(\frab\right)\cup V\left(\fra+\frab\right) \to  \Spec\left(R/\fra\cap \frab\right)
\]
is a universal topological epimorphism onto its image so it is an arc cover by \cite[Proposition 2.6]{BhattMathewArcTop} and \cite[Theorem 2.8]{RydhSubmersionsEffectiveDescent}.
Taking the sheafification in the arc topology and using \cite[Corollary 8.11]{BhattScholzePrisms} and \cite[Theorem 2.9]{BhattScholzeProjectivityWittGrassmannian}, we see that the above sequence remains exact after perfectoidization \ie the sequence
\[
0\to R/\left(\fra\cap \frab\right) \to R/\fra \oplus R/\frab \to \left(R/\left(\fra + \frab\right)\right)\perfd \to 0
\]
is exact. 
In particular, 
\[
R/\left(\fra + \frab\right) \to \left(R/\left(\fra + \frab\right)\right)\perfd
\]
must be an isomorphism.

\end{proof}

\begin{corollary}\label{perfectoidization is sum of perfectoidizations}
Let $R$ be a perfectoid ring and $\fra=(x_1,\ldots,x_n)$ a finitely generated ideal. Assume furthermore that each $x_i$ has a compatible system of $p$-power roots in $R$. Then, $\fra\perfd=(x_1)\perfd+\ldots+(x_n)\perfd=(\sqrt{x_1})^{-}+\ldots + (\sqrt{x_n})^{-}$.
\end{corollary}
\begin{proof}
The first equality is a direct consequence of \autoref{finite sum of perfectoid ideals is perfectoid}. The second one is \autoref{perfectoidization of principal ideal with roots}.
\end{proof}

\begin{proposition}[{\cf \cite[Lemma 2.4.3]{CaiLeeMaSchwedeTuckerPerfectoidSignature}}]
\label{perfectoid ideal under pcflat extension of pfd rings}
Let $R \to S$ be a $p$-completely flat morphism between perfectoid rings. Let $\fra=(h_1,\ldots,h_r)$ be a finitely generated ideal of $R$.
Then,
\[
(\fra S)\perfd=\sum_{i=1}^r \left((h_i R)\perfd S \right)^{-}=((\fra R)\perfd S)^-.
\]
\end{proposition}
\begin{proof}
By \autoref{finite sum of perfectoid ideals is perfectoid}, it suffices to show this in the case $\fra$ is principal, generated by $h\in R$. This was proved in \cite[Lemma 2.4.3]{CaiLeeMaSchwedeTuckerPerfectoidSignature}. 
\end{proof}

\subsection{Compatible ideals and centers of perfectoid purity} We are ready to define the main object of study of this paper. 

\begin{notation}\label{notation rainfty}
Let $(R,\m, k)$ be a complete Noetherian local ring with residue field of characteristic $p>0$. Fix a Cohen ring (a complete unramified mixed characteristic DVR with residue field $k$) $C_k\subset R$. Fix once and for all an inclusion $C_k\to W(k^\unpfty)$.
There is $A\coloneqq C_k\llbracket x_1,\ldots,x_d\rrbracket$, such that $R$ is a module-finite $A$-algebra and $C_k$ in $A$ maps to $C_k$ in $R$ (e.g. $A$ a Noether--Cohen normalization or $A\twoheadrightarrow R$). Let $A_\infty$ be the $p$-adic completion of 
\[
\left(A\widehat{\tensor}_{C_k} W\left(k^\unpfty\right)\right)\bigg[p^\unpfty,x_1^\unpfty,\ldots,x_d^\unpfty\bigg],
\]
which is perfectoid by \cite[Example 3.8 (4)]{BhattIyengarMaRegularRings}. Now, let 
\[
\rainfty\coloneqq (R\tensor_A A_\infty)\perfd,
\]
which is a perfectoid $R$-algebra by \cite[Theorem 10.11]{BhattScholzePrisms}. We write $\sA$ for the class of all $A=C_k\llbracket x_1,\ldots,x_d\rrbracket$ with $R$ a module finite $A$-algebra such that $C_k$ in $A$ maps to $C_k$ in $R$.
\end{notation}
\begin{remark}
Note that $A_\infinity$, and therefore $\rafty$, depend on the choice of a regular system of parameter for $A$. 
\end{remark}

\begin{definition}[{\cite{BMPSTWWPerfectoidPure}}]
Let $R,\ A$ be as in \autoref{notation rainfty}. We say that $R$ is perfectoid pure if the natural map $R\to \rafty$ is pure. By \cite[Lemma 4.23]{BMPSTWWPerfectoidPure}, this does not depend on the choice of $A$.
\end{definition}

\begin{definition}
In the setting of \autoref{notation rainfty}, let $\fra\subset R$ be an ideal and $\phi$ be in $\Hom_R(\rafty, R)$. We call the data of $(R,\phi)$ a \emph{pair}. 
Let 
\[\fra^A_\infty\coloneqq \left(\fra\rafty\right)\perfd.
\]
If $\phi\Big(\fra_\infty^A\Big)\subset \fra$, we say that 
$\fra$ is \emph{$\phi$-compatible}. If $B$ is any perfectoid $R$-algebra, we say that $\fra$ is \emph{$B$-compatible} if for all $\phi\in \Hom_R(B,R)$, $\phi((\fra B)\perfd)\subset \fra$. If $\fra$ is $\rafty$-compatible for all choices of $A\in \sA$, we say that $\fra$ is \emph{uniformly perfectoid compatible}.
\end{definition}

\begin{remark}
Since $R$ is complete, for any $R$-module $M$, a map $R\to M$ is pure if and only if it is split by \cite[Lemma 1.2]{FedderFPureRat}. This is something we often use, especially when dealing with a perfectoid pure $R$ since this gives us the existence of a pair $\rphi$ with $\phi$ surjective. 
\end{remark}
\begin{remark}
In \autoref{uniformly perfd compatible does not depend on choice of embedding}, we show that being uniformly perfectoid compatible does not depend on the choices of embeddings $C_k\subset R$ and $C_k\subset W(k^\unpfty)$. Moreover, if $k$ is perfect, there is no choice to be made since an embedding $C_k\subset R$ is equivalent to choosing a $p$-basis for $k$ and a lift of that $p$-basis to $R$ (see \cite[Theorem 23]{HochsterStructureCompleteLocalRings}).
This justifies that the name uniformly perfectoid compatible does not have a reference to the choices we made in \autoref{notation rainfty}.
\end{remark}

\begin{definition}
Let $\rphi$ be a pair with $\phi$ surjective. If $\p\in \Spec R$ is a $\phi$-compatible ideal, we say that $\p$ is a \emph{center of perfectoid purity of $\rphi$}.
Similarly, if $\p\in \Spec R$ is a uniformly perfectoid compatible ideal of $R$ then we say that $\p$ is a \emph{center of perfectoid purity of $R$} if $R$ is perfectoid pure.
\end{definition}

We show that our notion of uniformly perfectoid compatible ideal agrees with the classical notion of a compatible ideal when $R$ is of characteristic $p>0$, $F$-finite, and $F$-pure.
\begin{lemma}\label{lemma splitting from perfection}
Let $R$ be a reduced local ring of positive characteristic $p>0$ and assume that it is $F$-finite. Let $R\perf$ be the perfection of $R$ \ie $R\perf=\bigcup_e R^\unpe$. Let $f\in R$,
Then, the map $R\to\fstare R$, $1\mapsto \fstare f$ 
is split if and only if there is a map $R\perf\to R$ sending $f^\unpe$ to $1$. 
\end{lemma}

\begin{proof}
The backwards direction is straightforward as one can just restrict said splitting to $R^\unpe$ to get a splitting of $R\to R^\unpe\cong \fstare R$. Suppose, then, that $R$ is $e$-th Frobenius split along $f$, say by $\psi_1\colon R^\unpe\to R$ and let $\phi\coloneqq\psi_1\circ \left(\cdot f^\unpe\right)\colon R^\unpe\to R$. This is a splitting of the inclusion map $R\to R^\unpe$ (see e.g. \cite[Proposition 1.4.6]{SchwedeSmithFSingularitiesBook}). We define $\psi\colon R\perf\to R$ as follows: for $j>1$, let
\begin{align*} 
\psi_j \colon R^{1/p^{ej}} &\to R^{1/p^{e(j-1)}} \\ 
r^{1/p^{ej}} &\mapsto \left(\phi\left(r^\unpe\right)\right)^{1/p^{e(j-1)}}
\end{align*}
and
\[\psi \colon r^{1/p^{ei}} \mapsto \psi_1\circ \psi_2\circ\cdots\circ \psi_i\left(r^{1/p^{ei}}\right).
\]
We need to show that it is well defined \ie that $\psi\left(\left(r^{p^e}\right)^{1/p^{ei}}\right)=\psi\left(r^{1/p^{e\left(i-1\right)}}\right)$. This follows from the fact that 
\[
\psi_i\left(\left(r^{p^e}\right)^{1/p^{ei}}\right)=\phi\left(\left(r^{p^e}\right)^{1/p^{e}}\right)^{1/p^{e(i-1)}}=\phi(r)^{1/p^{e(i-1)}}=(r\phi(1))^{1/p^{e(i-1)}}
\]
which is equal to $r^{1/p^{e(i-1)}}$.
We also have that $\psi$ is an $R$-module homomorphism as it is the composition of $R$-module homomorphisms and is a splitting since $\psi(f^\unpe)=1$ by definition so we are done.
\end{proof}
\begin{remark}
Note that the proof implies that as long as $R$ is $F$-pure, for any $e>0$, every map in $\Hom_R(R^\unpe,R)$ extends to a map in $\Hom_R(R\perf,R)$.
\end{remark}

\begin{corollary}\label{corollary equiv comp with char p}
Let $R$ be as in \autoref{notation rainfty} and assume further that it has characteristic $p>0$, is $F$-finite, and $F$-pure. Then, $R\perf=\rafty$. Let $\fra\subset R$ be an ideal. It is uniformly compatible with the classical definition if and only if is also uniformly perfectoid compatible with our definition. Moreover, our notion of compatible ideal of pairs agrees with the classical notion when working with surjective maps. Indeed,
let $\psi_1\colon R^\unpe\to R$ be surjective, say $\psi_1(f^\unpe)=1$ for some $f\in R$. Let $\psi\colon R\perf\to R$ be the corresponding splitting as in \autoref{lemma splitting from perfection}. Then, $\fra$ is $\psi$-compatible if and only if it is $\psi_1$ compatible \ie $\psi_1(\fra^\unpe)\subset \fra$. 
\end{corollary}
\begin{proof}
The first part follows from the definitions and \autoref{lemma splitting from perfection} since $W(k^\unpfty)/p=k^\unpfty$. If $\psi(\fra\perf)\subset\fra$, then $\psi_1(\fra^\unpe)=\psi(\fra^\unpe)\subset\psi(\fra\perf) \subset \fra$. On the other hand if $\psi_1(\fra^\unpe)\subset \fra$,
letting $\psi_j$ and $\phi$ be as in \autoref{lemma splitting from perfection} for $j>0$, we have
\[\psi_j\left(\fra^{1/p^{ej}}\right)=\left(\phi\left(\fra^{1/p^{e}}\right)\right)^{1/p^{e(j-1)}}=\left(\psi_1\left(f^\unpe\fra^\unpe\right) \right)^{1/p^{e(j-1)}}\subset \fra^{1/p^{e(j-1)}}.\]
By induction, $\psi\left(\fra^{1/p^{ej}}\right)\subset \fra$ so $\psi(\fra\perf)\subset \fra$.
\end{proof}

\begin{proposition}\label{compatible equivalent to comm diag quotient}
Let $\rphi$ be a pair and $\fra\subset R$ be an ideal. Then, $\fra$ is $\phi$-compatible if and only if $\phi$ descends to a map $\bar{\phi}\colon(R/\fra)^A_\infty\to R/\fra$ \ie if and only if the diagram below commutes.
\begin{equation}\label{comm diag quotient comp ideal}  
\xymatrix{
\rafty \ar[d] \ar[r]^\phi & R \ar[d]\\
(R/\fra)^A_\infty\ar[r]^{\bar{\phi}}& R/\fra . 
}
\end{equation}
and the corresponding result also holds for uniformly perfectoid compatible ideals.
Moreover, there is a bijection between the $\phi$-compatible ideals of $R$ containing $\fra$ and the $\bar{\phi}$-compatible ideals of $R/\fra$. 
\end{proposition}
\begin{proof}
    The ideal $\fra$ is $\phi$-compatible if and only if the composition map 
    \[\rafty\xlongrightarrow{\phi} R\longrightarrow R/\fra
    \]factors through $(R/\fra)_\infty^A=\rafty/\aafty$. The second part of the statement follows from the isomorphism theorems.
\end{proof}

\begin{lemma}\label{p adic closure of comp ideal}
Let $(R,\fram, k)$ be a complete local Noetherian ring of residue characteristic $p>0$ and let $B$ be any $R$-algebra. Let $\phi\in \Hom_R(B,R)$ be arbitrary. Let $\fra\subset B$ and $\frab\subset R$ be ideals. Then, $\phi(\fra)\subset \frab$ if and only if $\phi(\bar{\fra})\subset \frab$ where $\bar{\fra}$ is the $p$-adic closure of $\fra$ in $B$.
\end{lemma}
\begin{proof}
We only need show the ``only if" direction. We know that $\frab$ is $p$-adically closed since $R$ is $p$-adically complete and Noetherian. Then, $\frab=\cap_{n\in \N} (\frab+ (p^n))$. Now, let $x\in \bar{\fra} = \cap_{n\in \N} (\fra+ (p^n))$. We have \[
\phi(x)\in  \bigcap_{n\in \N} \phi(\fra)+\phi((p^n))\subset \bigcap_{n\in \N} \frab+p^n\phi\left(B\right) \subset \bigcap_{n\in \N} (\frab+ (p^n))=\frab,
\]
as desired.
\end{proof}

\begin{proposition}[{\cf \cite[Lemma 5.1]{SchwedeCentersOfFPurity}}]\label{comp ideal of pair injective hull}
Let $\rphi$ be a pair and let $\fra\subset \rafty$ and  $\frab \subsetneq R$ be ideals. The following are equivalent:

\begin{enumerate}
    \item $\phi\left(\fra\right)\subset \frab$.
    \item For any $x\in \fra$, the composition 
    \[
    \rafty \xlongrightarrow{\times x} \rafty \xlongrightarrow{\phi} R \longrightarrow R/\frab
    \]
    is zero. 
    \item For any $x\in \fra$, the composition 
    \begin{align*}
&\Hom_R(R,R)\xlongrightarrow{\Hom_R(\phi,R)}\Hom_R\left(\rafty,R\right)\xlongrightarrow{\Hom_R(\times x, R)}\Hom_R\left(\rafty,R\right)\\
    & \longrightarrow \Hom_R(R,R)\cong R 
    \end{align*}
    has image in $\frab$. 
    \item[(d)] For any $x\in \fra$, the composition 
\[
E_{R/\frab} \longrightarrow E_R \longrightarrow E_R\tensor_R \rafty \xlongrightarrow{\id\tensor_R (\times x)} E_R\tensor_R \rafty \xlongrightarrow{\id \tensor \phi} E_R
\]
is zero, where $E_R$ is the injective hull of the residue field $k$ of $R$ over $R$ and $E_{R/\frab} $ is the injective hull of $k$ over $R/\frab$.
\end{enumerate}
and if it holds for $\fra=\frab^A_\infty$ then $\frab$ is $\phi$-compatible.
\end{proposition}
\begin{proof}
The equivalence of (a) and (b) readily follows from the definition. The implication (b) implies (c) is direct whereas if we do not have (a) then for some $x\in \aafty$, $\phi(x)\notin \frab$ so (c) does not hold. 
The equivalence between (d) and (c) is a standard application of Matlis duality. 
\end{proof}

A similar result holds for uniformly perfectoid compatible ideals.

\begin{proposition}[{\cf \cite[Lemma 5.1]{SchwedeCentersOfFPurity}}]\label{uniformly compatible injective hull}
Let $(R,\fram,k)$ be a complete Noetherian local ring and let $B$ be an $R$-algebra. 
Let $\fra\subset B$ and  $\frab \subsetneq R$ be ideals. The following are equivalent:
\begin{enumerate}
    \item $\fra$ gets sent to $\frab$ under all maps $\phi\in \Hom_R(B,R)$.
    \item For any $x\in \fra$ and $\phi \in \Hom_R(B,R)$ the composition 
    \[
    B \xlongrightarrow{\times x} B \xlongrightarrow{\phi} R \longrightarrow R/\frab
    \]
    is zero. 
    \item For any $x\in \fra$, the composition 
    \[
\Hom_R\left(B,R\right)\xlongrightarrow{\Hom_R(\times x, R)}\Hom_R\left(B,R\right) \longrightarrow \Hom_R(R,R)\cong R \to R/\frab
    \]
    is zero. 
    \item[(d)] For any $x\in \fra$, the composition 
\[
E_{R/\frab} \longrightarrow E_R \longrightarrow E_R\tensor_R B \xlongrightarrow{\id\tensor_R (\times x)} E_R\tensor_R B
\]
is zero, where $E_R$ is the injective hull of $k$ of $R$ over $R$ and $E_{R/\frab} $ is the injective hull of $k$ over $R/\frab$.
\end{enumerate}
If $R$ has residue characteristic $p>0$ and one (equivalently all) of these hold for $B$ a perfectoid $R$-algebra, $\fra=(\frab B)\perfd$, we have that $\frab$ is $B$-compatible. In particular, keeping the notation of \autoref{notation rainfty},
if this holds for $B=\rafty$ for all possible
choices of $A\in \sA$ and for $\fra=\frab^A_\infty$, we have that $\frab$ is uniformly perfectoid compatible.
\end{proposition}
\begin{proof}
Same as \autoref{comp ideal of pair injective hull}.  
\end{proof}

\subsection{New characterizations of uniform perfectoid compatibility} In this section, we show that uniform perfectoid compatibility can be checked on any choice of $A\in \sA$\footnote{
This proof was suggested by Karl Schwede and the author is very grateful to be able to include it here.} and that it does not depend on the embedding choices that we made in \autoref{notation rainfty}.
\smallskip

\noindent This next lemma is well known to experts but the author does not know a good reference. 

\begin{lemma}\label{pcomp fflat implies pure}
Let $(R,\fram,k)$ be a Noetherian local ring with residue characteristic $p>0$. If $A\to B$ is a $p$-completely faithfully flat map of $R$-modules and $E$ is the injective hull of $k$ over $R$, then $A\tensor_R E\to B\tensor_R E$ is injective. In particular, if $R$ is $\fram$-adically complete, $\Hom_R(B,R)\to \Hom_R(A,R)$ is surjective. Moreover, if $A=R$, then $A\to B$ is pure. 
\end{lemma}
\begin{proof}
We use the same technique as in the proof of \cite[Lemma 4.5]{BMPSTWWPerfectoidPure}.
We know that $A/p^n\to B/p^n$ is pure for every $n$. Let $E$ be the injective hull of $R/\fram$ over $R$. For every finitely generated submodule $N$ of $E$, $N$ is $p^n$-torsion for some n. Therefore, $A\tensor_R N \to B\tensor_R N$ can be identified as $A/p^n\tensor_R N \to B/p^n \tensor_R N$, which is injective by the purity of $A/p^n\to B/p^n$. By taking a direct limit over all such $N$, we find that $A\tensor_R E \to B\tensor_R E$ is injective. The surjectivity of $\Hom_R(A,R)\to \Hom_R(B,R)$ follows from Matlis duality.
The last statement follows by \cite[Proposition 6.11]{HochsterRobertsRingsOfInvariants}.
\end{proof}

\begin{lemma}\label{compatible iff pcomp fflat comp}
Let $(R,\fram,k)$ be a complete Noetherian local ring with residue characteristic $p>0$ and let $\fra\subset R$ be any ideal.
Let $B\to C$ be a $p$-completely faithfully flat morphism of perfectoid $R$-algebras. Then, $\fra$ is $B$-compatible if and only if it is $C$-compatible. 
\end{lemma}
\begin{proof}
By \autoref{pcomp fflat implies pure}, we have a surjection $\Hom_R(C,R)\to \Hom_R(B,R)$ and by \autoref{perfectoid ideal under pcflat extension of pfd rings}, $(\fra C)\perfd=((\fra B)\perfd C)^-$. Suppose that $\fra$ is $C$-compatible and let $\phi\in \Hom_R(B,R)$. There is $\psi\in \Hom_R(C,R)$ that extends $\phi$ to $C$. Then, 
\[
\phi((\fra B)\perfd)\subset \psi( (\fra B)\perfd) \subset \psi((\fra C)\perfd) \subset \fra
\]
so $\fra$ is $B$-compatible. Now, suppose that $\fra$ is $B$-compatible. By \autoref{p adic closure of comp ideal}, it suffices to show that $(\fra B)\perfd C$ gets sent to $\fra$ under any map in $\Hom_R(C,R)$. By \autoref{uniformly compatible injective hull}, it suffices to show that for any $x\in (\fra B)\perfd$, $y\in C$, the following composition is $0$:
\[
\Hom_R\left(C,R\right)\xlongrightarrow{\Hom_R(\times xy, R)}\Hom_R\left(C,R\right) \longrightarrow \Hom_R(R,R)\cong R \to R/\fra
\]
Since $R\to C$ factors through $R\to B$, the following diagram commutes:
\[\xymatrixcolsep{5pc}
\xymatrix{
\Hom_R(C,R)\ar^{\Hom_R(\times xy, R)}[r] \ar^{\Hom_R(\times y, R)}[d] &\Hom_R(C,R) \ar[r] \ar[d] &\Hom_R(R,R)\cong R \\
\Hom_R(B,R)\ar^{\Hom_R( x, R)}[r] &\Hom_R(B,R) \ar[ur]\\
}
\]
but by \autoref{uniformly compatible injective hull}, the composition through the bottom has image in $\fra$ so we are done.
\end{proof}

\begin{proposition}\label{B compatible can assume has roots and is complete}
Let $(R,\fram, k)$ be a complete Noetherian local ring with residue characteristic $p>0$ and $\fra\subset R$ be an ideal. Let $B$ be a perfectoid $R$-algebra. Then, there is a perfectoid $B$-algebra $C$ that contains a compatible system of $p$-th power roots of a given subset of elements of $R$ such that $\fra$ is $B$-compatible if and only if it is $C$-compatible. In fact, we can even assume that $C$ is absolutely integrally closed. We can also assume that $C$ is $\fram$-adically complete.
\end{proposition}
\begin{proof}
The first part follows from \autoref{compatible iff pcomp fflat comp} and André's flatness lemma \cite[Théorème 2.5.1]{AndreConjectureFacteurDirect}, \cite[Theorem 7.14]{BhattScholzePrisms}. The second part follows from the facts that $C\to \widehat{C}$ is faithfully flat and that $\widehat{C}$ is perfectoid \cite[Example 3.8]{BhattIyengarMaRegularRings}. 
\end{proof}

\begin{lemma}\label{rafty to rbfty is pure when adding roots}
Let $R$ and $A\in \sA$ be as in \autoref{notation rainfty}. Let $h_1,\ldots,h_r\in R$ be arbitrary. Let $B\coloneqq A\llbracket y_1,\ldots,y_r \rrbracket$ and make $R$ into a $B$-algebra by sending $y_i$ to $h_i$. Then, the natural morphism $\rafty\to \rbfty$ is $p$-completely faithfully flat. 
\end{lemma}
\begin{proof}
There are natural ring maps
\begin{equation}\label{maps rafty rbfty}
R \to R\tensor_A A_\infty \to R\tensor_A B_\infty \to R\tensor_B B_\infty \to \rbfty
\end{equation}
so by the universal property of perfectoidization, the map $R\to \rbfty$ factors through $\rafty$. By André's flatness lemma (\cite[Théorème 2.5.1]{AndreConjectureFacteurDirect}, \cite[Theorem 7.14]{BhattScholzePrisms}), there is a perfectoid $\rafty$-algebra, say $C$, such that $\rafty\to C$ is $p$-completely faithfully flat and $C$ contains a compatible system of $p$-th power roots for each of the $h_i$s. This gives a map $A_\infty$ to $C$ and then a map $B_\infty \to C$ by sending the $y_i^{\unpe}$s to the $\pe$-th power root of the $h_i$s in $C$ in a compatible way. We also have a natural map $R\to C$ and therefore maps from all the rings in \autoref{maps rafty rbfty} to $C$ that commute with each other. This gives us a factorization $R\to \rafty \to \rbfty \to C$ so the map $\rafty \to \rbfty$ is also $p$-completely faithfully flat.
\end{proof}

\begin{proposition}\label{compatible with rafty iff compatible with rbfty}
Let $R$ and $A\in \sA$ be as in \autoref{notation rainfty}. Let $\fra\subset R$ be an ideal and let $h_1,\ldots,h_r$ be arbitrary elements of $R$. Let $B\coloneqq A\llbracket y_1,\ldots,y_r \rrbracket$ and make $R$ into a $B$-algebra by sending $y_i$ to $h_i$. Then, $\fra$ is $\rafty$-compatible if and only if it is $\rbfty$-compatible.
\end{proposition}
\begin{proof}
This is direct from \autoref{rafty to rbfty is pure when adding roots} and \autoref{compatible iff pcomp fflat comp}.
\end{proof}

\begin{corollary}\label{uniformly compatible doesn't depend on the choice of A}
Let $R$ be as in \autoref{notation rainfty} and $\fra\subset R$ an ideal. Fix any $A\in \sA$. If $\fra$ is compatible with all maps in $\Hom_R(\rafty,R)$, then $\fra$ is uniformly perfectoid compatible. That is, one can test uniform perfectoid compatibility on only one $A\in \sA$.
\end{corollary}
\begin{proof}
Let $B$ be any other ring in $\sA$.
There is a regular local ring $C$ with 
\[
C=A\llbracket x_1,\ldots,x_r\rrbracket = B\llbracket y_1,\ldots,y_s\rrbracket
\]
for some $r,\ s\in \N$ such that the maps $A\to R$ and $B\to R$ factor through $C\to R$. Here, we are using that the Cohen ring $C_k$ for $k$ that we fixed is in $A$ and $B$ and maps to the same $C_k$ in $R$. 
By \autoref{compatible with rafty iff compatible with rbfty}, $\fra$ is $\rafty$-compatible if and only if it is $R^C_\infty$-compatible if and only if it is $R^B_\infty$-compatible.
\end{proof}

\begin{lemma}\label{not pfd pure p not in P implies compatible}
Let $R$ be as in \autoref{notation rainfty}, $\fra$ be an ideal with generators $h_1,\ldots,h_r$. Let $B$ be a perfectoid $R$-algebra that has a compatible system of $p$-th power roots for $h_1,\ldots, h_r$. We denote these roots by $\{x_{i,e}\}$ where $x_{i,e}^\pe=h_i$. If for all $e\gg 0, \ i=1,\ldots,r$, $\phi\in \Hom_R(B,R)$, $\phi(x_{i,e})\subset \fra$, then $\fra$ is $B$-compatible. In the special case where $\fra=\p$ is a prime ideal, it suffices to check that for all $e\gg 0, \ i=1,\ldots,r$, the map
\begin{align}\label{map to loc of rafty at p}
    R_\p &\to  B_\p\\
    1 &\mapsto x_{i,e} \nonumber
\end{align}
is not pure. In particular, if $B=\rafty$ for some $A\in \sA$, then this shows that $\fra$ is uniformly perfectoid compatible.
\end{lemma}
\begin{proof}
By \autoref{perfectoidization is sum of perfectoidizations}, we know that $(\fra B)\perfd=\left(\sum_{i=1}^r \sum_e (x_{i,e})\right)^{-}$. Suppose that $\fra$ is not compatible with $\phi\in \Hom_R(B,R)$. Using \autoref{p adic closure of comp ideal}, this means that
\[
\phi\left(\sum_{i=1}^r \sum_e \left (x_{i,e}\right) \right)\nsubset \fra.
\]
In particular, there must be $i$ and $e$ with $\phi( (x_{i,e}))\nsubset \fra$ so there is $y\in B$ with $\phi( yx_{i,e})\notin \fra$. By letting $\psi\coloneqq \phi(y\cdot -)$, we have that $\psi(x_{i,e})\notin \fra$, a contradiction. If $\fra=\p$ is a prime ideal then the above would imply that $\psi_\p$ is a splitting of the map $R_\p\to B_\p$ sending $1\to x_{i,e}$, also a contradiction. The last statement follows directly from \autoref{uniformly compatible doesn't depend on the choice of A}.
\end{proof}

\begin{proposition}\label{uniformly perfd compatible does not depend on choice of embedding}
Let $R$ be as in \autoref{notation rainfty} and let $\fra\subset R$ be any ideal. Then, $\fra$ is uniformly perfectoid compatible if and only if for every perfectoid $R$-algebra $B$, $\fra$ is $B$-compatible. In particular, for any ideal $\fra\subset R$, being uniformly perfectoid compatible does not depend on the choices of embeddings $C_k\into R$ and $C_k\into W(k^\unpfty)$.
\end{proposition}
\begin{proof}
We only need to show that if $\fra$ is uniformly perfectoid compatible, then for any perfectoid $R$-algebra $B$, $\fra$ is $B$-compatible. We show the contrapositive so let $B$ be any perfectoid $R$-algebra and suppose that there is $\phi\in \Hom_R(B,R)$ with $\phi( (\fra B)\perfd)\nsubset \fra$.
By \autoref{B compatible can assume has roots and is complete}, we can assume that $B$ is $\fram$-adically complete and has a compatible system of $p$-th power roots for all the elements of $R$. By \autoref{not pfd pure p not in P implies compatible}, we can assume that there is $x\in \fra$ with a $p^e$-th root $y\in B$ such that $\phi(y)\notin\fra$. By the proof of \cite[Lemma 4.23]{BMPSTWWPerfectoidPure}, for any $A\in \sA$, there is a map $A_\infty\to B$ making $B$ into an $A_\infty$-algebra that agrees with the map $R\to B$ when restricted to (the image of) $A$. Now, make $R$ into an $A\llbracket z\rrbracket$-algebra by sending $z$ to $x$. By sending the $p$-th power roots of $z$ in $A_\infty [ z^{\unpfty} ]$ to a compatible system of $p$-th power roots of $x$ in $B$ that has $y$ as its $p^e$-th root, we extend $A_\infty\to B$ to $A\llbracket z \rrbracket_\infty \to B$. By the universal properties of tensor products and perfectoidization, this gives a ring map $i\colon R^{A\llbracket z \rrbracket}_\infty \to B$. Let $\psi\coloneqq \phi\circ i\colon R^{A\llbracket z \rrbracket}_\infty\to R$. Denoting by $w$ the $\pe$-th root of $x$ in $\rafty$ that goes to $y$ in $B$, we have $\psi(z)\notin \fra$ so $\fra$ is not uniformly perfectoid compatible.
\end{proof}

\subsection{Intersections, sums, and associated primes} In this section we show that, just like in positive characteristic, compatible ideals behave well under basic ideal operations.

\begin{proposition}[{\cite[Lemma 4.29]{BMPSTWWPerfectoidPure}}]\label{perfectoid injective rings are reduced and WN}
If $R$ is perfectoid injective, in particular if $R$ is perfectoid pure, then $R$ is reduced and weakly normal.
\end{proposition}

\begin{corollary}\label{compatible ideals are radical}
Let $\rphi$ be a pair and $\fra\subset R$ a $\phi$-compatible ideal. If $\phi$ is surjective, then $\fra$ is radical. In particular, if $R$ is perfectoid pure, then the uniformly perfectoid compatible ideals of $R$ are radical.
\end{corollary}
\begin{proof}
The surjectivity of $\phi$ implies that of $\bar{\phi}\colon \rafty/\aafty\to R/\fra$ as in \autoref{compatible equivalent to comm diag quotient}.  
But $\bar{\phi}$ being surjective implies that $R/\fra$ is perfectoid pure so we are done by \autoref{perfectoid injective rings are reduced and WN}. The statement about uniformly perfectoid compatible ideals then follows from the fact that $R$ is complete so there is a splitting $\phi\colon \rafty \to R$ for some/any $A\in \sA$.
\end{proof}

\begin{proposition}\label{lemma compatible ideals closed under sums intersection}
Let $\rphi$ be a pair and $\{\fra_i\}_{i\in I}$ be $\phi$-compatible ideals. Then, $\cap_{i\in I} \fra_i$ and $\sum_{i\in I} \fra_i$ are $\phi$-compatible. The corresponding result also holds for uniformly perfectoid compatible ideals.
\end{proposition}
\begin{proof}
By \autoref{intersection of perfectoid ideals is perfectoid}, $\cap_{i\in I} (\fra_i)^A_\infty$ is perfectoid.
Then, 
\[\phi\left(\left(\bigcap_{i\in I} \fra_i\right)^A_\infty\right)\subset \phi\left(\bigcap_{i\in I} (\fra_i)^A_\infty\right)\subset \bigcap_{i\in I} \phi\left((\fra_i)^A_\infty\right)\subset \bigcap_{i\in I} \fra_i
\]so $\cap_{i\in I} \fra_i$ is $\phi$ compatible. Since we are in a Noetherian ring, any sum of ideals is finite. In particular, to show $\sum_{i\in I} \fra_i$ is $\phi$-compatible, it suffices to show that if $\fra$ and $\frab$ are two $\phi$-compatible ideals, then so is $\fra+\frab$. 
Using \autoref{finite sum of perfectoid ideals is perfectoid}, 
\[\phi\left(\left(\fra+ \frab\right)^A_\infty\right)\subset\phi\left(\aafty+ \bafty\right) = \phi\left(\aafty\right)+\phi\left(\bafty\right)\subset \fra + \frab
\]
so $\fra+\frab$ is also compatible.
\end{proof}

To show that compatible ideals of pairs are closed under associated primes, we will first need the corresponding statement about uniformly perfectoid compatible ideals. 

\begin{proposition}\label{uniform compatible closed under ass primes}
Let $R$ be as in \autoref{notation rainfty}. Let $\fra\subset R$ be a uniformly perfectoid compatible ideal. Then, the minimal primes $\q_1,\ldots,\q_r$ are also uniformly perfectoid compatible ideals. 
\end{proposition}
\begin{proof}
Without loss of generality, it suffices to show that $\q_1\eqqcolon \q$ is uniformly perfectoid compatible. Let $\{h_1,\ldots,h_s\}$ be a set of generators for $\q$. Let $A\in \sA$ be such that $\rafty$ has a compatible system of $p$-th power roots for the $h_i$s which we denote by
$\{z_{j,e}\}$ where $z_{j,e}^\pe=h_j$.
By \autoref{uniformly compatible doesn't depend on the choice of A}, it suffices to show that $\q$ is compatible with all $\phi\in \Hom_R(\rafty,R)$.
Fix $w$ in  $(\cap_{i>1} \q_i)\setminus \q$ and $\phi\in \Hom_R(\rafty,R)$. We have 
\[
w\phi\left(z_{j,e} \rafty\right)=\phi\left(wz_{j,e} \rafty\right)\subset \phi\left(\aafty\right)\subset \fra \subset \q.
\]
The first containment follows from noting that $\left(wz_{j,e}\right)^\pe=w^{\pe}h_j \in \fra$ and the fact that perfectoid ideals are radical. Since $\q$ is prime and $w\notin \q$, we must have that $\phi\left(z_{j,e}\rafty\right)\subset \q$. Now, \[
\q^A_\infty=\left( \sum_{j,e} \left(z_{j,e}\rafty\right) \right)^{-}
\]
so we are done by \autoref{p adic closure of comp ideal}.
\end{proof}

\begin{corollary}\label{ass primes of ring are uniformly compatible}
Let $R$ be as in \autoref{notation rainfty}. Then, the minimal primes of $R$ are uniformly perfectoid compatible.
\end{corollary}

\begin{proposition}\label{compatible ideals of pairs closed under ass primes}
Let $\rphi$ be a pair and $\fra\subset R$ a compatible ideal. Then, the minimal primes $\q_1,\ldots,\q_r$ of $\fra$ are also $\phi$-compatible. 
\end{proposition}
\begin{proof}
By \autoref{compatible equivalent to comm diag quotient}, $\phi$ descends to a map $\bar{\phi}\colon (R/\fra)^A_\infty\to R/\fra$. By \autoref{ass primes of ring are uniformly compatible}, the associated primes of $R/\fra$ are uniformly perfectoid compatible hence $\bar{\phi}$-compatible. But these are the images of the $\q_i$ in $R/\fra$ so by \autoref{compatible equivalent to comm diag quotient}, the $\q_i$s are $\phi$-compatible. 
\end{proof}

We now recall a classical result of Enescu and Hochster. 

\begin{proposition}\cite[Corollary 3.2]{EnescuHochsterTheFrobeniusStructureOfLocalCohomology}\label{finiteness result enescu hochster}
A family of radical ideals in an excellent local ring closed under sum,
intersection, and primary decomposition is finite.
\end{proposition}

\begin{corollary}\label{finiteness of phi compatible ideals}
Let $\rphi$ be a pair. If $\phi$ is surjective, then there are finitely many $\phi$-compatible ideals. 
\end{corollary}
\begin{proof}
This is immediate from \autoref{finiteness result enescu hochster},\autoref{compatible ideals of pairs closed under ass primes}, and \autoref{lemma compatible ideals closed under sums intersection}.
\end{proof}

\begin{corollary}\label{finiteness of uniformly compatible ideals}
Let $R$ be as in \autoref{notation rainfty} and assume further that it is perfectoid pure. Then, there are only finitely many uniformly perfectoid compatible ideals.
\end{corollary}
\begin{proof}
This follows from \autoref{finiteness of phi compatible ideals}.
\end{proof}
\begin{remark}\label{remark bound number cpps}
In fact, if $n$ is the embedding dimension of $R$ and $R$ is perfectoid pure, \cite[Theorem 4.2]{SchwedeTuckerNumberOfFSplit}, says that there are at most ${n\choose d}$ centers of perfectoid purity of $R$ of height $d$. Similarly for the centers of perfectoid purity of $\rphi$ for a surjective $\phi$.
\end{remark}

\section{Perfectoid Purity along Elements}

In this section, we study a variant of perfectoid purity and link it to compatible ideals.

\begin{definition}[{\cf \cite{RamanathanFrobeniusSplittingSchubert}, \cite{SmithGloballyFRegular}}]
Let $R$ be a Noetherian ring admitting a perfectoid algebra $B$, and let $x$ be an element of $B$. If the $R$-module map $R\to B$ sending $1$ to $x$ is pure, we say that $R$ is \emph{perfectoid pure along} $x$.
\end{definition}

We now state a well known fact about pure maps in our specific situation since we will repeatedly use it.

\begin{lemma}\label{pfd pure with smaller root}
Let $R$ be a Noetherian ring with $p$ in its Jacobson radical, $B$ a perfectoid $R$-algebra, and $C$ be a perfectoid $B$-algebra. Let $b\in B$ and $c\in C$ be such that the $R$-module map $R\to C$ sending $1$ to $c$ factors through the $R$-module map $R\to B$ sending $1$ to $b$. If $R$ is perfectoid pure along $c$, then it is perfectoid pure along $b$.
\end{lemma}

\begin{proposition}[\cf {\cite[Lemma 4.5]{BMPSTWWPerfectoidPure}}]\label{pfd pure can be with alg closed alg}
Let $R$ be a Noetherian ring with $p$ in its Jacobson radical, $r\in R$, and $B$ a perfectoid $R$ algebra. Suppose that for some $e\in \N_{>0}$, $r$ has an $\pe$-th root in $B$, which we denote by $x$. If $R$ is perfectoid pure along $x$, then there is a perfectoid $B$-algebra $B'$ that contains all the $p$-th power roots of $r$ such that $R$ is perfectoid pure along $x_{B'}$ where $x_{B'}$ is the image of $x$ in $B'$. In fact, we can even assume that $B'$ is absolutely integrally closed. Moreover, if $B$ already contains some other $p$-th power root $y$ of $r$ that is compatible with $x$, we can choose $B'$ that has a compatible system of $p$-th power roots of $y$ that is also compatible with $x$.
\end{proposition}
\begin{proof}
The proof of \cite[Lemma 4.5]{BMPSTWWPerfectoidPure} works \emph{mutadis mutandis} here. 
\end{proof}

\begin{proposition}[{\cf \cite[Lemma 4.8]{BMPSTWWPerfectoidPure}}]\label{pfd pure iff completion pfd pure}
Let $(R,\fram)$ be a Noetherian local ring of residue characteristic $p>0$, $r\in R$, $n\in \N_>0$. Then, there is a perfectoid $R$-algebra $B$ that contains an $n$-th root $x$ of $r$ such that $R$ is perfectoid pure along $x$ if and only if there is a perfectoid $\widehat{R}$ algebra $B'$ that contains an $n$-th root $x_{B'}$ of $r$ such that $R$ is perfectoid pure along $x_{B'}$.
\end{proposition}
\begin{proof}
Again, the proof is essentially the same as in \cite[Lemma 4.8]{BMPSTWWPerfectoidPure}. Suppose that there is a perfectoid $R$-algebra $B$ that contains an $n$-th root $x$ of $r$ with $R$ perfectoid pure along $x$. Then, by \cite[Proposition 2.1.11 (e)]{CesnaviciusScholzePurityFlatCohomology}, $\widehat{B}$ is a perfectoid $\widehat{R}$-algebra and the completion of the pure map $R\to B$, $1\mapsto x$ is pure: $E\coloneqq E(\widehat{R}/\fram)=E(R/\fram)\to E\tensor B \cong E\tensor \widehat{B}$ is injective. On the other hand, suppose that there is a perfectoid $\widehat{R}$-algebra $B'$ containing an $n$-th root $x_{B'}$ of $r$ with $\widehat{R}$ perfectoid pure along $x_{B'}$. Since the map $R\to \widehat{R}$ is faithfully flat, it is pure, so the composition map $R\to B'$, $1\mapsto x_{B'}$ is pure. 
\end{proof}

\begin{proposition}[\cf {\cite[Lemma 4.23]{BMPSTWWPerfectoidPure}}, {\cite[Proposition 2.9]{YoshikawaComputationMethodPerfectoidPurity}}]\label{pfd pure can take smaller algebra}
Let $R$ be as in \autoref{notation rainfty}. Let $h\in R$ be any element, $e\in \N_{>0}$, and suppose that there is a perfectoid $R$-algebra $B$ such that $h$ has a $p^e$-th root, say $x$, in $B$ such that $R$ is perfectoid pure along $x$. Then, there is a choice of $A\in \sA$ such that $\rafty$ has a $p^e$-th root of $x$, say $y$, and $R$ is perfectoid pure along $y$.
\end{proposition}
\begin{proof}
By \autoref{pfd pure can be with alg closed alg} and \autoref{pfd pure iff completion pfd pure}, we can assume that $B$ is absolutely algebraically closed and $\fram$-adically complete. By the proof of \autoref{uniformly perfd compatible does not depend on choice of embedding}, there is an $A\in \sA$ such that $\rafty$ has a $p^e$-th root $y$ of $h$ and 
$R\to B$ factors through $i\colon \rafty\to B$ with $i(y)=z$. The result then follows from \autoref{pfd pure with smaller root}.
\end{proof}

\begin{proposition}\label{max ideal unif compatible implies not pfd pure}
Let $(R,\fram)$ be as in \autoref{notation rainfty}. If $\m$ is a uniformly perfectoid compatible ideal of $R$, then for all $h\in \fram$ and $n\in \N_{> 0}$, there are no perfectoid $R$-algebra $B$ with an $n$-th root $x$ of $h$ such that $R$ is perfectoid pure along $x$.
\end{proposition}
\begin{proof}
Let $B$ be any perfectoid $R$-algebra.
The ideal $\fram$ contains $p$ so $(\fram B)\perfd=\sqrt{\fram B}$ by \autoref{perfectoidization is radical when p in ideal}. If $\fram$ is uniformly perfectoid compatible, then by \autoref{uniformly perfd compatible does not depend on choice of embedding} for any $\phi\in \Hom_R(B,R)$, 
\[
\phi\left(\sqrt{\fram B}\right)\subset \fram
\]
so for no root $x$ of an element of $\fram$, the map $R\to B$ sending $1$ to $x$ is split. Since $R$ is complete, this means that it is never pure by \cite[Lemma 1.2]{FedderFPureRat}.
\end{proof}


\begin{corollary}\label{loc is uniformly compatible then ideal is uniformly compatible}
Let $R$ be as in \autoref{notation rainfty}, $\p\in \Spec R$ with $p\in \p$.  If $\p \widehat{R}_\p$ is a uniformly perfectoid compatible ideal of $\widehat{R}_\p$ then $\p$ is a uniformly perfectoid compatible ideal of $R$. 
\end{corollary}
\begin{proof}
This follows from \autoref{max ideal unif compatible implies not pfd pure}, \autoref{pfd pure iff completion pfd pure}, \cite[Example 3.8(7)]{BhattIyengarMaRegularRings},
and \autoref{not pfd pure p not in P implies compatible}. 
\end{proof}

\section{Connection with Log Canonical Centers}

In this section, we prove that the log canonical centers of a perfectoid pure ring are centers of perfectoid purity. Many of the statements and proofs in this section are adaptations of the ones in \cite{BMPSTWWPerfectoidPure}.

\begin{definition}
Let $X$ be a normal Noetherian integral scheme with a dualizing complex $\omega^\bullet_X$ and a canonical divisor $K_X$. Let $\Delta$ be an effective divisor on $X$ with coefficients $\leq 1$.
We say that $(X,\Delta)$ is log canonical if $K_X+\Delta$ is $\Q$-Cartier and for every proper birational map $\pi\colon Y\to X$ with $Y$ normal, we have that the coefficients of $K_Y-\pi^*(K_X+\Delta)$ are $\geq -1$. Equivalently, for every proper birational map $\pi\colon Y\to X$ with $Y$ normal and reduced exceptional divisor $E$, we have that 
\begin{equation}\label{def lc}
\pi_*(\sO_Y(\lceil K_Y - \pi^*(K_X+\Delta) + E\rceil)=\sO_X.
\end{equation}
See for instance \cite[Corollary 2.31]{KollarMori}.
\end{definition}

\begin{remark}
If $X=\Spec R$ is affine with $R$ local and quasi-Gorenstein (so $K_X\sim 0$), then \autoref{def lc} is implied by $\pi_*(\omega_Y(E- \lfloor \pi^*(\Delta) \rfloor))\cong \omega_X$.
\end{remark}

\begin{definition}
Let $X=\Spec R$ be a normal Noetherian integral scheme with a dualizing complex $\omega^\bullet_X$ and a canonical divisor $K_X$. A log canonical center $Z\subset X$ is a closed subscheme of $X$ such that $X$ is log canonical at the generic point of $Z$ and for any $f\in \sI_Z$, the ideal of $R$ defining $Z$, and any $1\gg\varepsilon>0$, the pair $(X,\varepsilon \Div(f))$ is not log canonical at the generic point of $Z$. We will abuse notation and say $\sI_Z$ is a log canonical center of $R$.
\end{definition}

\begin{lemma}[\cf{\cite[Lemma 5.3]{BMPSTWWPerfectoidPure}}]\label{map to c bullet is pure}
Let $R, A$ be as in \autoref{notation rainfty}. Let $h\in R$ and assume that for some $y$ in a regular system of parameter of $A$, $y\mapsto h$. In particular, we can assume that $\rafty$ contains a compatible system of $p$-th power roots for $h$, which we denote by $\{z_e\}$ where $z_e^\pe=h$ for $e\in \N_{>0}$. Fix $e\in \N_{>0}$ and let $S\coloneqq R[x]/x^\pe-h$. 
Let $\pi\colon Y\to X\coloneqq\Spec S$ be a birational map. Let $Z$ be the subset of $X$ outside of which $\pi$ is an isomorphism and set $E\coloneqq \pi\inv(Z)_{\red}$. Let $\Delta\coloneqq \Div x$ on $Y$ and $C^\bullet$ be the following pullback in the (unbounded) derived $\infty$-category of $S$-modules
\[
\xymatrix{
C^\bullet \ar[d] \ar[r] & \rgamma(Y,\sO_Y) \ar[d] \\
\rgamma(Z, \sO_Z) \ar[r]& \rgamma(E, \sO_E).
}
\]
Then, the map $R\to \rafty$, $1\mapsto z_e$ factors as 
\[
R\to S \to C^\bullet \to \rafty
\]
where the first map is multiplication by $x$ and the second is the natural map $S\to C^\bullet$.
\end{lemma}
\begin{proof}
We need to show that $C^\bullet$ maps to $\rafty$. The (ring) morphism $R\to \rafty$ extends to a morphism $S\to \rafty$ by sending $x\to z_e$.
By \cite[Lemma 5.3]{BMPSTWWPerfectoidPure}, there is a map $C^\bullet$ to $S\perfd$. By the universal property of perfectoidization, there is a map $S\perfd\to \rafty$ and the result follows.
\end{proof}

\begin{lemma}[{\cf \cite[Proposition 5.15]{BMPSTWWPerfectoidPure}}]\label{lemma lcness can be tested from S}
Let $(R,\fram)$ be a Noetherian normal quasi-Gorenstein domain with a dualizing complex $\omega_R^\bullet$. Let $h\in R$ be arbitrary and let $S\coloneqq R[x]/x^\pe-h$ for some $e>0$. Note that $S$ is $S_2$ and quasi-Gorenstein.
Let $f\colon \Spec S\to \Spec R$ be the induced map. Suppose that for any birational $\mu\colon Y\to \Spec S$ with 
\begin{enumerate}
    \item $\mu$ is an isomorphism outside a set $V(J)\subset \Spec S$ of codimension $\geq 2$
    \item $Y$ is $G1$ and $S2$
    \item If $F=\mu\inv(V(J))_{\red}$, we have that $F$ has pure codimension $1$ and that $Y$ is regular at each generic point of $F$ (that is, $F$ can be viewed as a divisor),
\end{enumerate}
we have that the composition
\begin{equation}\label{map omegay(F) to omega(R) via x}
f_*\mu_*( x\cdot \sHom_Y(\sI_F,\omega_Y)) \to f_* (x\cdot \omega_S) \to \omega_R
\end{equation}
induced by the map $R\to S$, $1\mapsto x$ surjects.
Then, if $\lfloor \unpe\Div h \rfloor$ is $0$ on $R$, the pair $(R,\unpe \Div h)$ is log canonical.
\end{lemma}
\begin{proof} This is essentially the same proof as \cite[Proposition 5.15]{BMPSTWWPerfectoidPure}.
Let $\pi\colon X\to \Spec R$ be a blowup with $X$ normal and which is an isomorphism over $U=X\setminus \pi(D)$ for some (exceptional) divisor $D$ on X. Let $I\subset R$ be an ideal whose blowup produces $\pi\colon X\to \Spec R$. Let $Y_0\to \Spec S$ denote the blowup of $IS$ and note we have a finite map $Y_0\to X$. Let $V\subset Y_0$ be the inverse image of $U$ and note that it is quasi-Gorenstein since it is an open subset of $\Spec S$. Let $i\colon V\to Y_0$ be the inclusion. Let $\sC$ be the integral closure of $\sO_{Y_0}$ in $i_*\sO_U$ \ie $\sC=\sO_{Y_0}^N\cap i_* \sO_V$ where the intersection is taken in the fraction field of $Y_0$. Let $Y\coloneqq \textbf{Spec}_{Y_0}(\sC)$.
Then, $Y$ is $G1$ and $S2$ and has a finite map to $Y_0$ since our base is excellent. Therefore, there is a finite map $g\colon Y\to X$ and the induced map $Y\to \Spec S$ is an isomorphism over $V$. Let $E$ and $F$ denote the reduced exceptional sets of the maps $X\to \Spec R$ and $Y\to \Spec S$, respectively. By \cite[Theorem 5.14, Proposition 5.15]{BMPSTWWPerfectoidPure}, $F$ is a divisor in the sense of \cite{KollarSingulaitieofMMP}. We have the following commutative diagram
\[
\xymatrix{
F \ar@{^{(}->}[d] \ar@{->>}[r]^h & E \ar@{^{(}->}[d]\\
Y \ar[d]^\mu \ar@{->>}[r]^g & X \ar[d]^\pi\\
\Spec S\ar@{->>}[r]^f& \Spec R
}
\]
where all the horizontal maps are finite, by construction. This induces the following diagram of canonical modules
\[
\xymatrix{
0 \ar[r] & g_*\omega_Y \ar[d] \ar[r] & g_*\omega_Y(F) \ar[d]  \ar[r] & h_*\omega_F \ar[d]\\
0 \ar[r]& \sO_X(K_X) \ar[r] & \sO_X(K_X+E) \ar[r]& \omega_E\\
}
\]
where the notation $\sHom_Y(\sI_F,\omega_Y)=\omega_Y(F)=\sO_Y(K_Y+F)$ is reasonable as $Y$ is $G_1$ and $S_2$.
\begin{claim}
The image of $g_*(x\cdot \omega_Y(F))\to \sO_X(K_X +E)$ is contained in the sheaf $\sO_X(K_X + E - \lfloor \unpe \Div h \rfloor)$.
\end{claim}
\begin{proof}
Since all sheaves are $S2$, we can check this in codimension $1$. On $V$, we can reduce to the affine case and there is nothing to show. At the generic points of $F$, $Y$ is normal and the result follows from our choice of rounding.
\end{proof}
By pushing forward to $R$, we get a map 
\[
\pi_*g_*(x\cdot \omega_Y(F))\to \pi_*(\sO_X(K_X+E-\lfloor\unpe\Div h \rfloor))\to R\cong \omega_R.
\]
But this can also be factored as 
\[
f_*\mu_*( x\cdot \sHom_Y(\sI_F,\omega_Y)) \to f_* (x\cdot \omega_S) \to \omega_R,
\]
which we assumed to be surjective. Then, $\pi_*( \sO_X(K_X+E-\lfloor\unpe\Div h \rfloor))\to \omega_R$ is surjective and since $X$ was arbitrary, $R$ is log canonical.
\end{proof}
\begin{remark}\label{rk weaker assumptions for lcness from S}
Note the proof actually shows that it suffices to check \autoref{map omegay(F) to omega(R) via x} for all birational maps $Y\to \Spec S$ built from a blowup $\pi\colon X\to \Spec R$ with $X$ normal as in the proof.
\end{remark}

\begin{lemma}[{\cf \cite[Claim 5.5]{BMPSTWWPerfectoidPure}}]\label{bottom coh of C bullet}
Let $R$, $h$, $e$, and $S$ be as in \autoref{map to c bullet is pure}.
Let $X\to \Spec R$ be a blowup with $X$ normal. Let $\omega_R^\bullet$ be a dualizing complex for $R$ and let $\bold{D}$ denote Grothendieck duality $\RHom_R(-,\omega_R^\bullet)$.
Constructing $Y$ as in \autoref{lemma lcness can be tested from S} and $C^\bullet$ from the birational map $\mu\colon Y\to \Spec S$, as in \autoref{map to c bullet is pure}, we have 
\[H^{-d}(\bold{D}(C^\bullet))=\Gamma(Y,\omega_Y(F))\] for $F$ the reduced exceptional divisor of $\mu$.
The corresponding map $\Gamma(Y,\omega_Y(F) )\to \omega_R$ factors through $\Gamma(Y,x\cdot \omega_Y(F))$. Here, we interpret $\omega_Y(F)$ as in \autoref{lemma lcness can be tested from S}.
\end{lemma}
\begin{proof}
This follows from \cite[Claim 5.5]{BMPSTWWPerfectoidPure} and the choice of map $R\to C^\bullet$.
\end{proof}

Before our next lemma, we need a result about pure maps in the derived category, which we state for the convenience of the reader. 

\begin{lemma}[{\cite[Proposition 2.11]{BMPSTWWPerfectoidPure}}]\label{pure implies inj on coh}
Let $R$ be a Noetherian ring, $\fra\subset R$ an ideal, and let $f\colon M\to N$ be a pure map in $D(R)$ in the sense of \cite[Section 2.1]{BMPSTWWPerfectoidPure}. Then, $H^i \bold{R}\Gamma_\fra M\to H^i \bold{R}\Gamma_\fra N$ is injective for all $i$.
\end{lemma}

\begin{proposition}[{\cf \cite[Proposition 5.4]{BMPSTWWPerfectoidPure}, \cite{KovacsSchwedeSmithLCImpliesDuBois}}]\label{lc implies log canonical p in p}
Let $R$, $A$, $h$, $e$ be as in \autoref{map to c bullet is pure}. Let $\p\in \Spec R$ be a prime that contains $p$. 
If $R$ is quasi-Gorenstein and $e$ is such that $\lfloor \unpe\Div h \rfloor$ is $0$ on $R$, the map $R_\p\to R^A_{\infty,\p}$ sending $1$ to $h^\unpe$ is pure, then $(R_\p,\unpe \Div h)$ is log canonical. 
\end{proposition}
\begin{proof} 
By taking the $\p$-adic completion of $R_\p$ and $R^A_{\infty,\p}$ and using \autoref{pfd pure iff completion pfd pure}, we can reduce to the case where $\p$ is the maximal ideal.
Let $C^\bullet$ be as in \autoref{bottom coh of C bullet}.
The map $R\to C^\bullet$ is pure by \autoref{map to c bullet is pure} so $H_\fram^d(R)\to H_\fram^d(\rgamma(X,C^\bullet))$ is injective by \autoref{pure implies inj on coh}. 
Then,
the dual map $f_* \mu_*(x\cdot \omega_Y(F))\to \omega_X$, surjects for $f$ and $\mu$ as in \autoref{lemma lcness can be tested from S} and the result follows from \autoref{lemma lcness can be tested from S}.
\end{proof}

\begin{proposition}[{\cf \cite[Proposition 5.4]{BMPSTWWPerfectoidPure}, \cite{KovacsSchwedeSmithLCImpliesDuBois}}]\label{lc implies log canonical}
Let $R$, $A$, $h$, $e$, $S$, and $x$ be as in \autoref{map to c bullet is pure}. Let $\p\in \Spec R$ be a prime that does not contain $p$ and assume that $h\in \p$.
If $R$ is quasi-Gorenstein, $e\in \N_{>0}$ is such that $\lfloor \unpe\Div h \rfloor$ is $0$ on $R$, and 
the map $R_\p\to R^A_{\infty,\p}$ sending $1$ to $h^\unpe$ is pure, then the pair $(R_\p,\unpe \Div h)$ is log canonical. 
\end{proposition}
\begin{proof}
Let $\pi_\p\colon X_\p\to \Spec R_\p$ be a blowup, say of the ideal $I\subset R_\p$ and let $\pi\colon X \to \Spec R$ be the blowup of $\Spec R$ at the ideal $I\cap R$. Assume that both $X$ and $X_\p$ are normal. 
Keeping the notation of \autoref{lemma lcness can be tested from S} applied to $\pi$ and construct $S$, $C^\bullet$ as in \autoref{map to c bullet is pure} and $Y$, $F$ with maps $f\colon \Spec S\to \Spec R$ and $\mu\colon Y\to \Spec S$
as in \autoref{lemma lcness can be tested from S}. 
Let $C^\bullet_\p\coloneqq C^\bullet\tensor_R R_\p$. Note that $C^\bullet_\p$ is the following pullback in the (unbounded) derived $\infty$-category of $S_\p$-modules 
\[
\xymatrix{
C_\p^\bullet \ar[d] \ar[r] & \rgamma(Y_\p,\sO_{Y_\p}) \ar[d] \\
\rgamma(Z_\p, \sO_{Z_\p}) \ar[r]& \rgamma(F_\p, \sO_{F_\p}).
}
\]
where $X_\p$, $Y_\p$ and $F_\p$ are the base changes of $X$, $Y$ and $F$, respectively, from $\Spec R$ to $\Spec R_\p$. Let $\mu_\p\colon Y_\p\to \Spec S_\p$ and $f_\p\colon \Spec S_\p \to \Spec R_\p$ be the corresponding maps.
Since the proof of \autoref{bottom coh of C bullet} (\cf \cite[Claim 5.5]{BMPSTWWPerfectoidPure}) did not make use of the $p$-completeness of $R$, letting $\omega^\bullet_{R_\p}$ be a normalized dualizing complex for $R_\p$ and $\bold{D}$ denote Grothendieck duality $\RHom_R(-,\omega^\bullet_{R_\p})$, we have 
\[
H^{-r}(\bold{D}(C^\bullet))=\Gamma(Y_\p,\omega_{Y_\p}(F_\p))
\]
for $r=\dim R_\p$.
Since the multiplication-by-$x$ map $R_\p$ to $C_\p$ is pure, by \autoref{pure implies inj on coh}, the corresponding map $H^r_{\p R_\p}(R_\p)\to H^r_{\p R_\p}(C^\bullet_\p)$ injects and its dual surjects. Then, the map 
\[
(f_\p)_* (\mu_\p)_*(x \cdot \omega_{Y_\p}(F_\p))\to \omega_{R_\p}
\] 
surjects and by \autoref{lemma lcness can be tested from S} and \autoref{rk weaker assumptions for lcness from S}, we are done. 
\end{proof}

\begin{theorem}[{\cf \cite[Theorem 6.7]{SchwedeCentersOfFPurity}}]\label{lc center is cpp for big enough A}
Let $R$ be as in \autoref{notation rainfty}  and assume further that it is quasi-Gorenstein and normal. Let $\p\in \Spec R$ be a log canonical center of $R$. Then, $\p$ is uniformly perfectoid compatible. In particular, if $R$ has no uniformly perfectoid compatible ideal, it must be klt.
\end{theorem}
\begin{proof}
By \autoref{not pfd pure p not in P implies compatible}, it is enough to show that for $A\in \sA$ with a subset of a system of parameters mapping to generators $\{h_1,\ldots,h_r\}$ of $\fram$, the maps $R\to R^A_{\infty}$, $1\mapsto h_i^\unpe$ are not pure for $e\gg 0$. Since none of the pairs $(R_\p, \Div(h^{\unpe}))$ are log canonical, the maps $R_\p\to R^A_{\infty,\p}$, $1\mapsto h^\unpe$ are not pure by \autoref{lc implies log canonical p in p} and \autoref{lc implies log canonical} so we are done.
\end{proof}

We believe that the following lemma well known to experts but we could not find a reference. 
\begin{lemma}\label{multiplier ideal is intersection lc centers}
Let $R$ be a normal $\Q$-Gorenstein log canonical ring and suppose that it has finitely many log canonical centers. Then, the multiplier ideal $\sJ\subset R$ is an intersection of ideals of log canonical centers of $R$. 
\end{lemma}
\begin{proof}
Let $f\colon Y\to X\coloneqq \Spec R$ be a proper birational map. Using that $R$ is log canonical, we can write $\ceil{K_Y-f^* K_X}$ as $\sum a_i E_i - \sum b_j F_j$ where each $E_i$, $F_j$ is a prime divisor,
$a_i>0$ and $b_j=1$. Then, 
\[\Gamma(Y,\sO_Y(\ceil{K_Y-f^* K_X}))=\Gamma\left(Y,\sO_Y\left( - \sum F_j\right)\right)=\bigcap_j \Gamma\left(Y,\sO_Y\left( - F_j\right)\right)
\]
which is a finite intersection of ideals of log canonical centers. Taking the intersection over all proper birational maps $f\colon Y\to X$ gives the desired result.
\end{proof}

\begin{corollary}\label{multiplier ideal uniformly compatible}
Let $R$ be as in \autoref{notation rainfty}  and assume further that it is perfectoid pure, quasi-Gorenstein and normal. Then, the multiplier ideal $\sJ\subset R$ is uniformly perfectoid compatible. 
\end{corollary}
\begin{proof}
Since $R$ is perfectoid pure, it has finitely many log canonical centers by \autoref{lc center is cpp for big enough A}, \autoref{finiteness of uniformly compatible ideals} so we are done by \autoref{multiplier ideal is intersection lc centers}, \autoref{lc center is cpp for big enough A}, and \autoref{lemma compatible ideals closed under sums intersection}.
\end{proof}

\begin{remark}
Based on the positive characteristic result \cite{SchwedeCentersOfFPurity}, we expect this to hold even when $R$ is not quasi-Gorenstein. A first step could be to generalize this to the $\Q$-Gorenstein case with index not divisible by $p$ as in \cite[Corollary 5.11]{BMPSTWWPerfectoidPure}.
\end{remark}

\section{Normality}

In this section, we show that the conductor ideal is uniformly perfectoid compatible, implying that the presence of compatible ideals detects (the failure of) normality. We use this to deduce various properties of perfectoid pure rings.

\begin{lemma}\label{lemma perfectoidization same ideal}
Let $R$ and  $A$ be as in \autoref{notation rainfty} and $R^N$ be the normalization of $R$. Let $\fracc$ be the conductor ideal of $R$ \ie the largest ideal of $R$ that is also an ideal of $R^N$. Then, $\cafty$ is an ideal both of $\rafty$ and $(R^N\tensor_A A_\infty)\perfd\eqqcolon\ranfty$.
\end{lemma}
\begin{proof}
We first show that $R\tensor_A A_\infty, \fracc\tensor_A A_\infty$, and $R^N\tensor_A A_\infty$ are all derived $p$-complete $A_\infty$-modules. Note that $R$, $\fracc$, and $R^N$ are all finitely presented $A$-modules since $R$ is excellent hence N-1 and Noetherian. This implies that $R\tensor_A A_\infty, R^N\tensor_A A_\infty, $ and $\fracc\tensor_A A_\infty$  are all finitely presented $A_\infty$-modules. Since a finitely presented module over a derived $p$-complete ring is derived $p$-complete (see, e.g., \cite[Corollary 6.3.2]{KedlayaPrismaticNotes}), we are done.
Now, by \cite[Corollary 8.12]{BhattScholzePrisms}, we have a pullback diagram 
\[
\xymatrix{
\rafty \ar[d] \ar[r] & \ranfty \ar[d]\\
((R/\fracc)\tensor_A A_\infty)\perfd\ar[r]& \left(\left(R^N/\fracc\right)\tensor_A A_\infty \right)\perfd 
}
\]
which we claim implies that $\cafty$ is both an ideal of $\rafty$ and $R^{N,A}_\infty$. Indeed, the diagram can be rewritten as 
\[
\xymatrix{
\rafty \ar[d] \ar[r] & \ranfty \ar[d]^\pi\\
\rafty/\cafty \ar[r]_i & (R^N/\fracc)^A_\infty.
}
\]
Since this is a pullback diagram, we have that 
\[
\rafty=\left\{(x,y)\in \rafty/\cafty \times \ranfty \mid i(x)=\pi(y) \right\}
\]
where the map $\rafty\to \rafty/\cafty$ is the projection onto the first coordinate.
In particular, 
\[\cafty=\left\{(0,y)\in \rafty/\cafty \times \ranfty \mid \pi(y)=0 \right\}.\] Then, the image of $\cafty$ in $\ranfty$ is exactly the kernel of the quotient $\ranfty\to (R^N/\fracc)^A_\infty$, as desired.
\end{proof}

\begin{proposition}[{\cf \cite[Exercise 1.2.4(E)]{BrionKumarFrobeniusSplitting}}]\label{conductor ideal is compatible}
Let $\rphi$ be a pair and $\fracc\subset R$ be the conductor ideal. Then, $\fracc$ is $\phi$-compatible
for all maps $\phi\in \Hom_R(\rafty,R)$. In particular, if there are no non-trivial $\phi$-compatible ideals, $R$ is normal. 
\end{proposition}
\begin{proof}
By \autoref{lemma perfectoidization same ideal}, if $s\in R^N$ then $s\cafty\subset \cafty$ where we abuse notation and think about all these elements as part of $R^N_\infty$. Then, if $W$ is the set of nonzero divisors of $R$ and 
\[
\phi_W\coloneqq \phi\tensor_R \id_{W\inv R}\colon \rafty\tensor_R W\inv R\to R\tensor_R W\inv R,
\]
seeing $R$ as a subset of $W\inv R\cong R\tensor_R W\inv R$ via $r\to r\tensor_R 1$ gives
\[
s\phi(\cafty)=s \phi_W(\cafty)= \phi_W(s \cafty) \subset \phi_W(\cafty)= \phi(\cafty)
\]
so $\phi(\cafty)\subset \fracc$ hence $\fracc$ is compatible. Now, since $\fracc=\Ann_R(R^N/R)$, it is generically $K$ hence nonzero. In particular, if $R$ has no nontrivial compatible ideals, $\fracc=R$ so $R$ is normal.
\end{proof}
\begin{remark}\label{conductor uniformly compatible}
Note that this is not dependent on the choice of $\phi$ (nor on the choice of $A\in \sA$) so the conductor is uniformly perfectoid compatible. In particular, if $R$ has no uniformly perfectoid compatible ideals then it is normal. 
\end{remark}

\begin{remark}
It is well known that 
\[\fracc=\Im\left(\Hom_R\left(R^N,R\right)\xrightarrow{\phi\mapsto \phi(1)} R\right).\]
This is a specific case of a trace ideal, which are known to be compatible in positive characteristic (see for instance \cite[Lemma 2.2]{PolstraSchwedeCompatibleIdealsGorensteinRings}). It is a natural question to ask if this is also true in mixed characteristic. That is, if $R\to S$ is a finite extension of complete local Noetherian rings, is 
\[\Im\left(\Hom_R(S,R)\xrightarrow{\phi\mapsto \phi(1)} R\right)\] a uniformly perfectoid compatible ideal of $R$? Unfortunately, it is not true in this generality. Indeed, let $R\coloneqq \Z_p\llbracket x,y,z\rrbracket/x^3+y^3+z^3$ with $p\equiv 1 \mod 3$. This is a perfectoid pure ring by \cite[Example 7.3]{BMPSTWWPerfectoidPure} since going modulo $p$ is the $(x,y,z)$-adic completion of the cone over an ordinary elliptic curve. Then, all the uniformly perfectoid compatible ideals have to be radical by \autoref{compatible ideals are radical}. Let $S=R$ with map $R\to S$ induced by multiplication by $p^n$ on the elliptic curve. 
By \cite[Example 4.14]{BMPSTWWMinimalModelProgramMixedCharacteristic}, the image of the trace map is the ideal $(p^n,x,y,z)$ which is not radical and therefore not compatible. In fact, even its radical, $(p,x,y,z)$, is not uniformly perfectoid compatible. Indeed, since $R/p$ is $F$-pure, the perfectoid pure threshold of $(R,\Div(p))$ is $1$ by \cite[Remark 2.10]{YoshikawaComputationMethodPerfectoidPurity}. In particular, there is a perfectoid $R$-algebra $B$ such that the map $R\to B$, $1\mapsto p$ is pure so $p$ is not in any uniformly perfectoid compatible ideal. 
\end{remark}

\begin{corollary}[{\cf \cite[Proposition 7.11]{SchwedeCentersOfFPurity},
\cite[Exercise 1.2.E(4)]{BrionKumarFrobeniusSplitting}}]\label{purity extends to normalization}
Let $\phi\colon \rafty\to R$ be an $R$-linear map and $R^N$ be the normalization of $R$. If $R$ is reduced, $\phi$ has a unique extension $\phi^N\colon \ranfty \to R^N$. In particular, $(R^N,\phi^N)$ is perfectoid pure if so is $(R,\phi)$.
\end{corollary}
\begin{proof}
Let $a\in \fracc$ be a nonzerodivisor and $K$ be the total ring of fraction of $R$.
We define $\phi^N\colon \ranfty \to K$ as
$\phi^N
(x)\coloneqq \phi(xa)/a$. We first show that this is an $R^N$-linear map. Let $W$ be the set of nonzero divisors of $R$, $\phi_W\coloneqq \phi\tensor_R \id_{W\inv R}$, and $s,\ x\in R^N$. Then, 
\[
s \phi^N(x)=s\frac{\phi(xa)}{a}=s\frac{\phi_W(xa)}{a}=\frac{\phi_W(xas)}{a}=\frac{\phi(xsa)}{a}=\phi^N(sx).
\]
Moreover, for $x,y\in \ranfty$, 
\[
\phi^N(x+y)=\frac{\phi(xa+ya)}{a}=\frac{\phi(xa)+\phi(ya)}{a}=\frac{\phi(xa)}{a}+\frac{\phi(ya)}{a}=\phi^N(x)+\phi^N(y).
\]

We want to show that the image of $\phi^N$ lands in $R^N$. For any $c\in \fracc$, 
\[c\phi^N(x)=\phi^N(cx)\in \fracc\subset R.\] More generally, 
\[
c\left(\phi^N(x)\right)^n=c\phi^N(x)\left(\phi^N(x)\right)^{n-1}\subset \fracc \left(\phi^N(x)\right)^{n-1}
\] 
which implies $\fracc\left(\phi^N(x)\right)^n\subset \fracc\left(\phi^N(x)\right)^{n-1}$, and so by induction $\fracc\left(\phi^N(x)\right)^n\subset \fracc\subset R$. If $R$ is a domain then by \cite[Proposition 2.4.8]{HunekeSwansonIntegralClosure}, $\phi^N(x)$ is integral over $R$, as desired. Else, let $\p_1,\ldots,\p_n$ be the minimal primes of $R$. By \autoref{uniform compatible closed under ass primes}, each $\p_i$ is uniformly compatible so by \autoref{compatible equivalent to comm diag quotient}, $\phi$ descends to a map $\phi_i\colon (R/\p_i)^A_\infty\to R/\p_i$.
Since $a$ was chosen to not be a zero divisor, it is not in any of the $\p_i$s. Moreover, for any $x\in (R/\p_i)^N$, using the fact that 
\[R^N=(R/\p_1)^N\times \cdots \times (R/\p_n)^N\]
(see \eg \cite[Corollary 2.1.13]{HunekeSwansonIntegralClosure}),
we see that $ax\in R/\p_i$. In particular, the above reasoning can be used to extend $\phi_i$ to a map $\phi_i^N(R/\p_i)^{N,A}_\infty \to (R/\p_i)^N$.
By the proof of \autoref{intersection of perfectoid ideals is perfectoid}, we have that $\ranfty=\prod_i (R/\p_i)^A_\infty$. We claim that the extension we are looking for is $\psi\coloneqq \prod_i \phi^N_i$. Since $R$ is reduced, we have an inclusion $R\into R/\p_i\times \cdots \times R/\p_n$. 
For $r\in \rafty$, 
\[
\psi(r)=\prod_i \phi^N_i(r)=\prod_i \phi_i(ar)/a=\prod_i \phi(r) \mod{\p_i},
\]
so $\psi(r)$ is precisely the image of $\phi(r)$ in $\prod_i R/\p_i$, as wanted.
It remains to show the uniqueness of such map. Suppose $\phi^N_1$ and $\phi^N_2$ are both extensions of $\phi$ to $\ranfty$. Then, for any $c\in \fracc$ a nonzerodivisor, 
\[
c\phi^N_1(x)=\phi^N_1(xc)=\phi(xc)=\phi^N_2(xc)=c\phi^N_2(x)
\]
and since $c$ is a nonzerodivisor, we must have $\phi^N_1(x)=\phi^N_2(x)$. In particular, the definition of $\phi^N$ does not depend on the chosen $a\in \fracc$.
\end{proof}

We can also use the compatibility of the conductor ideal to show that perfectoid pure rings are (WN1). When $R$ has char $p>0$ and is $F$-split, this was shown by Schwede and Zhang in \cite{SchwedeZhangBertiniTheoremsForFSings}. The mixed characteristic proof follows their strategy.

\begin{definition}[{\cite{CuminoManaresiSingularitiesWeaklyNormal}}] 
Let $(R,\fram)$ be a reduced local weakly normal ring. We say that $(R,\fram)$ is (WN1) if the normalization morphism $R\to R^N$ is unramified in codimension $1$. That is, for every prime ideal $\q$ of height $1$ in $R^N$, and $\p=\q\cap R$, $\p R^N_\q=\q R^N_\q$ and $R_\p/\p R_\p\subset R^N_\q/\q R^N_\q$ is separable.
\end{definition}

\begin{proposition}[{\cf \cite[Theorem 7.3]{SchwedeZhangBertiniTheoremsForFSings}}]
Let $\rphi$ be a pair and suppose that $\phi$ is surjective. Then, $R$ is \emph{(WN1)}. In particular, if $R$ is a perfectoid pure complete local Noetherian ring, then it is \emph{(WN1)}.
\end{proposition}
\begin{proof}
By \autoref{perfectoid injective rings are reduced and WN}, we only need to show that $R\to R^N$ is unramified in codimension $1$.
Localization commutes with normalization so we may assume without loss of generality that $(R,\fram)$ is local of dimension $1$ and so $R^N$ is semilocal of dimension $1$. Note that $R$ (and therefore $R^N$) may now have characteristic $0$. Let $(S,\fran)$ be the localization of $R^N$ at one of its maximal ideals (which has to lie over $\fram$). Since $R$ is perfectoid pure, the conductor is radical in $R$ and $R^N$ by \autoref{compatible ideals are radical}, \autoref{conductor ideal is compatible}, and \autoref{purity extends to normalization}. In particular, $\fracc=\fram$ and $\fracc S=\fram S$ is also radical and therefore must be equal to $\fran$. It remains to show that $R/\fram\to S/\fran$ is separable.
If $R/\fram$ has characteristic $0$, there is nothing to show so assume it has characteristic $p>0$ \ie $p\in \fram=\fracc$. We have a commutative diagram
\[
\xymatrix{
\ranfty/\cafty  \ar[r] & R^N/\fracc \\
\rafty/\cafty \ar[u]\ar[r]& R/\fracc \ar[u] 
}
\]
and since $p\in \fracc$, this can be rewritten as
\[
\xymatrix{
(R^N/\fracc)\perf  \ar[r] & R^N/\fracc \\
(R/\fracc)\perf   \ar[u]\ar[r]& R/\fracc \ar[u]. 
}
\]
Notice that $(R/\fracc)^\unp\subset(R^N/\fracc)^\unp $ so restricting the maps from the bottom left to $(R/\fracc)^\unp$ and localizing $(R^N/\fracc)^\unp\to R^N/\fracc$ at $\fran\cap R^N$
gives the following diagram

\[
\xymatrix{
(S/\fran)^\unp \ar[r] & S/\fran\\
(R^N/\fracc)^\unp  \ar[r]\ar[u] & R^N/\fracc \ar[u]\\
(R/\fracc)^\unp  \ar[u]\ar[r]& R/\fracc . \ar[u] 
}
\]
The horizontal maps are surjective hence nonzero which implies that $R/\fracc\to S/\fran$ is separable by \cite[Example 5.1]{SchwedeTuckerTestIdealFiniteMaps}. 
\end{proof}

\section{Splitting Prime}
In this section, we give an explicit description of the largest uniformly perfectoid compatible ideal of a ring $R$ and show that it detects the perfectoid purity of $R$, analogously to the splitting prime of Aberbach and Enescu \cite{AberbachEnescuStructureOfFPure}. We also generalize the idea to get a compatible core of an ideal, analogous to the Cartier core of \cite{BadillaCespedesFInvariantsSRRings,BrosowskyCartierCoreMap,CarvajalFayolleTameRamificationCFPS}.

\subsection{The positive characteristic case}
We first start by expressing the positive characteristic splitting prime in terms of the perfection of $R$.

\begin{definition}[{\cite{AberbachEnescuStructureOfFPure}}]
Let $(R,\fram)$ be a local ring of char $p>0$. Suppose further that $R$ is $F$-finite and let $\phi\in \Hom_R\left(R^\unpe,R\right)$. The splitting prime of $\rphi$ is 
\[
\upbeta\rphi=\bigcap_{n>0}\left\{r\in R\mid \phi^n\left(r^{1/p^{en}}R^{1/p^{en}}\right)\subset \fram\right\}
\]
where $\phi^n$ is defined as the composition
\[
R^{1/p^{en}}\xrightarrow[]{\phi^{1/p^{e(n-1)}}}R^{1/p^{e(n-1)}}\xrightarrow[]{\phi^{1/p^{e(n-2)}}}R^{1/p^{e(n-2)}}\xrightarrow[]{\phi^{1/p^{e(n-3)}}}\cdots \xrightarrow[]{\phi} R.
\]
\end{definition}
In the special case $\phi(1)=1$, we are able to express $\upbeta\rphi$ in terms of $R\perf$. Indeed,
\[
\upbeta\rphi = \left\{r\in R \Biggm| \bigcup_{n>0}\phi^n\bigg(r^{1/p^{en}}R^{1/p^{en}}\bigg)\subset \fram\right\}.
\]
Moreover, since 
\[
(r)\perf=\bigcup_{n>0} r^{1/p^{en}}R^{1/p^{en}}, 
\]
letting $\psi$ be the map from the proof of \autoref{lemma splitting from perfection} constructed only from $\phi$, we have 
\[
\upbeta\rphi=\left\{r\in R\mid \psi((r)\perf)\subset \fram\right\}.
\]

\subsection{The mixed characteristic case} 
Let $\rphi$ be a pair. The above discussion would lead us to try to define the splitting prime of $\rphi$ as
\[
\upbeta\rphi\coloneqq \left\{r\in R\mid \phi\left((r)_\infty^A\right)\subset \fram\right\},
\]
at least when $\phi(1)=1$.
Unfortunately, it is not clear to the author whether such an ideal is $\phi$-compatible. This brings us to our actual definition. 

\begin{definition}\label{definition splitting prime}
Let $\rphi$ be a pair. Let $\upbeta_0\rphi\coloneqq \fram$ and
\[
\upbeta_i\rphi\coloneqq \left\{r\in R\mid \phi\left((r)^A_\infty\right)\subset\upbeta_{i-1}\right\}
\]
for $i>0$. We then define the splitting prime $\upbeta\rphi$ as
\[
\upbeta\rphi\coloneqq \bigcap_{i>0}\upbeta_i\rphi.
\]
\end{definition}

\begin{remark}
When $\phi$ is surjective, we can show that $\upbeta_i\rphi\supset\upbeta_{i+1}\rphi$: let $r\notin \upbeta_i\rphi$ and $x\in \rafty$ such that $\phi(x)=1$. Then, $r=\phi(rx)\in \phi\left((r)_\infty^A\right)$ so $r$ is not in $\upbeta_{i+1}\rphi$, which shows the desired inclusion. Although it is not clear whether this inclusion is strict or not, it explains why we take the intersection over all the $\upbeta_i$s. Importantly, this inclusion is not strict in positive characteristic. In particular, this definition agrees with the positive characteristic one when $\phi$ is surjective.
\end{remark}

\begin{proposition}\label{basic ppties splitting prime}
With notation as in \autoref{definition splitting prime}, $\upbeta\rphi\neq R$ if and only if $\phi$ is surjective and, in that case, $\upbeta\rphi$ is the largest $\phi$-compatible ideal of $R$.
\end{proposition}
\begin{proof}
We first show that $\upbeta\rphi$ is $\phi$-compatible:
\begin{align*}
\phi\left((\upbeta\rphi)^A_\infty\right)&=\phi\left(\left(\bigcap_{i>0} \upbeta_i\rphi\right)^A_\infty\right)\\
&\subset \phi\left(\bigcap_{i>0} \left(\upbeta_i\rphi\right)^A_\infty\right) \\
&\subset \bigcap_{i>0} \phi\left( \left(\upbeta_i\rphi\right)^A_\infty\right) \\
&\subset \bigcap_{i\geq 0} \upbeta_i\rphi \\
&\subset \upbeta\rphi
\end{align*}
where the first containment follows from \autoref{intersection of perfectoid ideals is perfectoid}.
Let $\fra\subsetneq R$ be a $\phi$-compatible ideal and suppose that $\fra\subset\upbeta_i\rphi$. Then,
\[
\phi\left(\fra^A_\infty\right)\subset \fra\subset \upbeta_i\rphi
\]
so $\fra\subset\upbeta_{i+1}$. Since $\fra\subset\fram=\upbeta_0\rphi$, we must have 
\[
\fra\subset \bigcap_{i>0}\upbeta_i\rphi= \upbeta\rphi
\]
so $\upbeta\rphi$ is indeed the largest $\phi$-compatible ideal. It remains to show that $\upbeta\rphi\neq R$ if and only if $\phi$ is surjective. Notice that $\upbeta\rphi\neq R$ if and only if $\upbeta_1\rphi\neq R$. The backwards direction is clear so suppose that $\upbeta_1\rphi=R$. Then, we must have $\upbeta_2\rphi=R$ hence $\upbeta_3\rphi=R$ and so on. Now, $\phi$ is surjective if and only if $\phi\left((1)^A_\infty\right)=R$ if and only if $1\notin \upbeta_1\rphi$ if and only if $\upbeta\rphi\neq R$.
\end{proof}

This construction indeed defines a prime. 

\begin{corollary}\label{splitting prime is prime}
With notation as in \autoref{definition splitting prime}, if $\phi$ is surjective, then $\upbeta\rphi$ is a prime ideal. In particular, it is the largest perfectoid pure center of $\rphi$ and $R/\upbeta\rphi$ is a normal domain.
\end{corollary}
\begin{proof}
By \autoref{compatible equivalent to comm diag quotient}, the pair $(R/\upbeta\rphi, \bar{\phi})$ has no $\bar{\phi}$-compatible ideals. By \autoref{conductor ideal is compatible} $R/\upbeta\rphi$ must be a normal local ring so it is a domain.  
\end{proof}

\begin{definition}\label{definition uniform splitting prime}
Let $R$ and $A\in\sA$ be as in \autoref{notation rainfty}.
We define the splitting prime of $R$ as 
\[
\upbeta(R)=\bigcap_{\phi\in \Hom_R\left(\rafty,R\right)}\upbeta\rphi
\]
and write $\upbeta$ when there is no confusion on the ring. 
\end{definition}

\begin{proposition}\label{basic properties uniform splitting prime}
With notation as in \autoref{definition uniform splitting prime}, $\upbeta(R)\neq R$ if and only if $R$ is perfectoid pure, in which case it is the largest uniformly perfectoid compatible ideal of $R$. Moreover, $R/\upbeta$ has no uniformly perfectoid compatible ideal and is a normal domain.
\end{proposition}
\begin{proof}
It all readily follows from \autoref{basic ppties splitting prime}, \autoref{splitting prime is prime}, and \autoref{uniformly compatible doesn't depend on the choice of A}.
\end{proof}

\subsection{Compatible cores and a test ideal}\label{section compatible cores}
We can can generalize the idea of splitting prime to find the largest compatible ideal contained inside a given ideal, similar to the Cartier core construction of \cite{BadillaCespedesFInvariantsSRRings,BrosowskyCartierCoreMap,CarvajalFayolleTameRamificationCFPS}.

\begin{definition}
Let $\rphi$ be a pair and $\fra\subset R$ be any ideal. We define $\upbeta_\fra\rphi$, the compatible core of $\fra$ as 
\[
\upbeta_\fra\rphi=\bigcap_{i>0}\upbeta_{\fra,i}\rphi
\]
where $\upbeta_{\fra,0}\rphi=\fra$ and 
\[
\upbeta_{\fra,i}\rphi=\left\{r\in R\mid \phi\left((r)^A_\infty\right)\subset \upbeta_{\fra,i-1}\rphi\right\}
\]
for $i>0$.
\end{definition}

\begin{proposition}[{\cf \cite[Corollary 3.14]{BrosowskyCartierCoreMap}, \cite[Proposition 4.5]{CarvajalFayolleTameRamificationCFPS}}]\label{basic ppties compatible core}
Let $\rphi$ be a pair and $\fra\subset R$ be a radical ideal. 
If for all minimal primes $\p$ of $\fra$, $\Im(\phi)\nsubset \p$, then 
$\upbeta_\fra\rphi$ is the largest $\phi$-compatible ideal contained in $\fra$. 
\end{proposition}
\begin{proof}
The proof is the same as \autoref{basic ppties splitting prime}.
\end{proof}

\begin{proposition}[{\cf \cite[Theorem 3.3]{BrosowskyCartierCoreMap}, \cite[Proposition 4.14]{CarvajalFayolleTameRamificationCFPS}}]\label{compatible core as a map}
Let $\rphi$ be a pair. If $\sU\coloneqq\Spec R\setminus V(\Im\phi)$ and $\p\in \sU$, then $\upbeta_\p\rphi$ is a prime ideal and the map $\upbeta_{\rphi}\colon \sU \to \sU$ given by 
\[
\upbeta_{\rphi}\colon \p \mapsto \upbeta_\p\rphi
\]
is continuous in the Zariski topology.
\end{proposition}
\begin{proof}
Since $\p\in \sU$, $\upbeta_\p\rphi$ is radical. Indeed, $\sqrt{\upbeta_\p\rphi}\subset \p$ and is $\phi$-compatible by \autoref{lemma compatible ideals closed under sums intersection} and \autoref{compatible ideals of pairs closed under ass primes}. Now, all the the minimal primes of $\upbeta_\p\rphi$ are also compatible and at least one of them must be contained in $\p$. By \autoref{basic ppties compatible core}, this implies that $\upbeta_\p\rphi$ must be prime. To show that it is continuous in the Zariski topology, we follow the proof of \cite[Theorem 3.23]{BrosowskyCartierCoreMap}. Let $\fra\subset R$ be an ideal. We show that the inverse image of $V\coloneqq V(\fra)\cap \sU$ under $\upbeta_{\rphi}$ is also closed. In fact, we claim 
\[
\upbeta_{\rphi}\inv(V)= V(\frab)\cap \sU
\]
where
\[
\frab\coloneqq \bigcap_{\p\in \sU ,\upbeta_\p\rphi\in V} \upbeta_{\rphi}(\p).
\]
Indeed, if $\p\in \upbeta_{\rphi}\inv(V)$, $\upbeta_\p\rphi\subset \p$ and since $\frab\subset \upbeta_\p\rphi$, $\frab\subset \p$. On the other hand, since $\fra\subset\frab$, if $\p\in V(\frab)\cap \sU$ then $\fra\subset\frab\subset \upbeta_\p\rphi\subset \p$ so $\p\in V(\fra)\cap \sU$.
\end{proof}

Not only is this map continuous but we can actually describe the fibers explicitly when $\phi$ is surjective. However, we first need to define one more object.
Let $\rphi$ be a pair with $\phi$ surjective. By \autoref{finiteness of phi compatible ideals}, there are only finitely many $\phi$-compatible ideals. In particular, if $\p$ is a compatible prime, there are only finitely many $\phi$-compatible ideals not contained in that prime. The intersection of all these is therefore a nonzero ideal which is the smallest $\phi$-compatible ideal not contained in $\p$. This leads us to our next definition.

\begin{definition}[{\cf \cite{TakagiHigherDimensionalAdjoint, SmolkinSubadditivity}}]
Let $\rphi$ be a pair with $\phi$ surjective. Let $\p$ be a compatible prime of $\rphi$. We define $\uptau_\p\rphi$ to be the smallest $\phi$-compatible ideal not contained in $\p$ and call it the test ideal along $\p$. If $\p=0$, we write $\uptau\rphi$ and call it the test ideal of $\rphi$. If $R$ is not a domain, then we let $\uptau\rphi$ be the smallest $\phi$-compatible ideal not contained in any of the minimal primes of $R$.
\end{definition}

\begin{proposition}[{\cf \cite[Proposition 4.20]{CarvajalFayolleTameRamificationCFPS}}]\label{beta is locally closed}
Let $\rphi$ be a pair with $\phi$ surjective and let $\p\in \Spec R$ be a $\phi$-compatible prime. Then, $\upbeta\inv_{\rphi}(\p)=V(\uptau_\p\rphi)^\mathsf{C}\cap V(\p)$.
\end{proposition}
\begin{proof}
Let $\q\in \upbeta\inv_{\rphi}(\p)$. If $\uptau_\p\rphi\subset \q$, then by \autoref{basic ppties compatible core}, 
\[\uptau_\p\rphi\subset \upbeta_\q\rphi=\upbeta_\p\rphi,\]
a contradiction. On the other hand, if $\q\in V(\uptau_\p\rphi)^\mathsf{C} \cap V(\p)$ then $\p\subset \upbeta_\q\rphi$. If this were a strict inequality, we would have $\upbeta_\q\rphi\supset \uptau_\p\rphi$, a contradiction.
\end{proof}

As usual, this can also be done for uniformly perfectoid compatible ideals.

\begin{definition}
Let $R$ be as in \autoref{notation rainfty} and fix $A\in \sA$. Let $\fra\subset R$ be any ideal. We define $\upbeta_\fra(R)$, the compatible core of $\fra$ as 
\[
\upbeta_\fra(R)=\bigcap_{\phi\in \Hom_R\left(\rafty,R\right)} 
\upbeta_\fra\rphi
\]
\end{definition}

\begin{proposition}[{\cf \cite[Corollary 3.14]{BrosowskyCartierCoreMap}}]
Let $R$ be as in \autoref{notation rainfty} and fix $A\in \sA$. Let $\fra\subset R$ be a radical ideal. 
If for all minimal primes $\p$ of $\fra$ and $\phi\in \Hom_R\left(\rafty,R\right)$, $\Im(\phi)\nsubset \p$, then  
$\upbeta_\fra(R)$ is the largest $\phi$-compatible ideal contained in $\fra$. 
\end{proposition}
\begin{proof}
This follows from \autoref{basic ppties compatible core}.
\end{proof}

\begin{proposition}[{\cf \cite[Theorem 3.3]{BrosowskyCartierCoreMap}}]
Let $R$ be as in \autoref{notation rainfty} and fix $A\in \sA$. Let 
\[
\sU\coloneqq \bigcup_{\phi\in \Hom_R\left(\rafty,R\right)} \Spec R\setminus V(\Im\phi)
\]
and $\p\in \sU$. 
Then, $\upbeta_\p(R)$ is a prime ideal and the map $\upbeta_{R}\colon \sU \to \sU$ given by 
\[
\upbeta_{R}\colon \p \mapsto \upbeta_\p(R)
\]
is continuous in the Zariski topology.
\end{proposition}
\begin{proof}
Same as \autoref{compatible core as a map}
\end{proof}

\begin{definition}[{\cf \cite{TakagiHigherDimensionalAdjoint, SmolkinSubadditivity}}]
Let $R$ be as in \autoref{notation rainfty} and assume further that it is perfectoid pure. Let $\p$ be a uniformly perfectoid compatible prime of $R$. We define $\uptau_\p(R)$ to be the smallest uniformly perfectoid compatible ideal not contained in $\p$ and call it the test ideal along $\p$. If $\p=0$, we write $\uptau(R)$ and call it the test ideal of $R$. If $R$ is not a domain, then we let $\uptau(R)$ be the smallest uniformly perfectoid compatible ideal not contained in any of the minimal primes of $R$.
\end{definition}

\begin{proposition}[{\cf \cite[Proposition 4.20]{CarvajalFayolleTameRamificationCFPS}}]\label{fibers of beta are tau}
Let $R$ be as in \autoref{notation rainfty} and assume further that it is perfectoid pure. Let $\p$ be a uniformly perfectoid compatible prime of $R$. Then, $\upbeta\inv_{R}(\p)=V(\uptau_\p(R))^\mathsf{C}\cap V(\p)$.
\end{proposition}
\begin{proof}
    Same as \autoref{beta is locally closed}.
\end{proof}

\begin{lemma}\label{trace does not depend on choice of A}
Let $A\subset R$ be a Noether normalization and let $h,g\in A$ be arbitrary. Let $\lambda_{g,A}(R)\coloneqq \sum_{\phi\in \Hom_R\left(\rafty,R\right)}\phi\left((g)^A_\infty\right)$. Let $B\coloneqq A\llbracket y\rrbracket$ and make $R$ and $A$ into $B$-algebras by sending $y$ to $h$. Then, $\lambda_{g,A}(R)=\lambda_{g,B}(R)\coloneqq \sum_{\phi\in \Hom_R\left(\rbfty,R\right)}\phi\left((g)^A_\infty\right)$
\end{lemma}
\begin{proof}
By \autoref{rafty to rbfty is pure when adding roots}, we have a surjection $\Hom_R(\rbfty,R)\to \Hom_R(\rafty,R)$ and the inclusion $\subset$ follows. For the other inclusion, fix $\psi\in \Hom_R(\rbfty,R)$, $x\in (g)^A_\infty$ and $z\in \rbfty$. Let $\phi\colon \rafty\to R$ be the composition of the maps $\times z\colon \rafty\to \rbfty$ and $\psi$. Then, $\psi(xz)=\phi(x)\in \lambda_{g,A}$. This implies that $\psi((g)^A_\infty \rbfty)\subset \lambda_{g,A}$. By \autoref{p adic closure of comp ideal} and \autoref{perfectoid ideal under pcflat extension of pfd rings} we get $\psi((g)^B_\infty)\subset \lambda_{g,A}$ as desired.
\end{proof}

\begin{remark}
The hope would be that $\uptau(R)$ is equal to other mixed characteristic test ideals (see \cite{MaSchwedePerfectoidTestIdeal,HaconLamarcheSchwedeGlobalGenTestIdeal,BMPSTWWMinimalModelProgramMixedCharacteristic,BMPSTWWDTestIdealPerturbation,MurayamaUniformBoundsSymbolicPowers,RobinsonBCMTestIdealToric,SatoTakagiDeformationFPureFreg,PerezRGCharacteristicFreeTestIdeals}). In this generality, this is far beyond the scope of this paper. However, if $R$ is normal $\Q$-Gorenstein and $A\subset R$ is a Noether normalization, by \cite[Lemma 5.1.6]{CaiLeeMaSchwedeTuckerPerfectoidSignature}, we are are able to describe the BCM-test ideal $\tau_{\rafty}(R)$ as 
\[
\tau_{\rafty}(R)=\sum_{\phi\in \Hom_R\left(\rafty,R\right)}\phi\left((g)^A_\infty\right)
\]
where $g\in A$ is such that $A[g\inv]\to R[g\inv]$ is étale. 
It is not clear to the author if such an ideal is compatible. Interestingly, by \autoref{trace does not depend on choice of A}, for a fixed $g\in R$, the ideal $\sum_{\phi\in \Hom_R\left(\rafty,R\right)}\phi\left((g)^A_\infty\right)$ does not depend on the choice of $A\in \sA$ which is hinting at it being uniformly perfectoid compatible. 
Let $\uptau_1(R,A)\coloneqq \tau_{\rafty}(R)$, 
\[\uptau_i(R,A)\coloneqq \sum_{\phi\in \Hom_R\left(\rafty,R\right)}\phi\left(\uptau_{i-1}(R,A)^A_\infty\right)\]
for $i>1$ and 
\[
J\coloneqq \sum_{i,A} \uptau_i(R,A)
\]
then $J$ is a nonzero uniformly perfectoid compatible ideal. It is equal to $\uptau(R)$ if and only if there is $x\in \uptau(R)$ with $A[x\inv]\to R[x\inv]$ étale. Since for any $x\in \uptau(R)$, $xg$ has that property, we see that $J=\uptau(R)$. In particular, $\tau_{\rafty}(R)\subset \uptau(R)$ and a normal $\Q$-Gorenstein BCM-regular ring has no uniformly perfectoid compatible ideals. If $R$ is perfectoid pure, \autoref{multiplier ideal uniformly compatible} gives us that after inverting $p$ we have containments
\begin{equation}\label{testtestmult}
\tau_{\rafty}[1/p]\subset \uptau(R)[1/p]\subset \sJ(R[1/p]),
\end{equation}
where $\sJ(R[1/p])$ is the multiplier ideal of $R[1/p]$.
When $R$ furthermore satisfies the hypotheses of \cite[Theorem B]{BMPSTWWDTestIdealPerturbation} (that is, $R$ is a flat Gorenstein domain and the completion of a ring of finite type over a DVR), the first and third ideals of \autoref{testtestmult} are equal and therefore our test ideal also agrees with them after inverting $p$.
\end{remark}

\section{Behavior Under Étale Morphisms}\label{section behavior étale morphisms}

In this section, we show that compatible ideals behave well under étale morphisms. This relies heavily on the almost purity theorem of Bhatt--Scholze, which we now state. Other versions of the almost purity theorem can be found in \cite{FaltingsAlmostEtaleExtensions,ScholzePerfectoidSpaces,KedlayaLiuRelativeHodgeTheory,AndreLemmeAbhyankarPerfectoide}.

\begin{theorem}\cite[Theorem 10.9]{BhattScholzePrisms}
Let $R$ be a perfectoid ring and $\fra$ a finitely generated ideal of $R$. Let $S$ be a finitely presented finite $R$-algebra such that $\Spec S\to \Spec R$ is finite étale outside $V(\fra)$. Then, $S_{\perfdd}$ is discrete and a perfectoid ring and the map $S\to S_{\perfdd}$ is an isomorphism away from $V(\fra)$. In particular, a finite étale cover of a perfectoid ring is perfectoid. 
\end{theorem}

Now, let $\rphi$ be as in \autoref{notation rainfty} and $(S,\fran)$ be an $R$-algebra such that $R\to S$ is finite étale. By stability under base change, $\rafty \tensor_{R} S $ is finite étale over $\rafty \tensor_{R} R \cong \rafty$, in particular it is perfectoid. In fact, it is isomorphic to $S^A_\infty$ by the universal properties of perfectoidization and tensor product. This gives us the following setting. 

\begin{setting}\label{setting étale morphism}
Let $(R,\fram)$ be a complete local Noetherian ring. Let $(S,\fran)$ be an $R$-algebra such that $R\to S$ is finite étale. By \cite[Lemma 4.6, Lemma 4.15]{BMPSTWWPerfectoidPure}, $R$ is perfectoid pure if and only if $S$ is. Take any unramified regular local ring $A$ such that $R$ (and therefore $S$) is a module finite $A$-algebra.
Fix $\phi\in \Hom_R(\rafty,R)$.
Let $\psi\coloneqq \phi \tensor_{R} S\colon S^A_\infty\to S$ so $\psi$ is an extension of $\phi$ to $\safty$.
We have a commutative diagram 
\begin{equation}\label{comm diag extension of pairs}  
\xymatrix{
\safty  \ar[r]^\psi & S \\
\rafty\ar[r]^\phi \ar[u]& R \ar[u]. 
}
\end{equation}
where the vertical arrows are inclusions. The data $\rphi\to \spsi$ as above is what we now call an étale morphism of pairs. 
\end{setting}

\begin{lemma}\label{lemma extension of perfectoid ideal étale case}
Let $R\to S$ be as in \autoref{setting étale morphism} and $\fra\subset R$ be an ideal. If $\fra^{A,R}_\infty$ is the perfectoidization of $\fra$ in $\rafty$, then $\fra^{A,R}_\infty\tensor_R S\eqqcolon\fra^{A,S}_\infty$ is the perfectoidization of $\fra$ in $\safty$.
\end{lemma}
\begin{proof}
Note that $\rafty/\fra^{A,R}_\infty\tensor_R S$ is finite étale over $\rafty/\fra^{A,R}_\infty\tensor_R R= R/\fra^{A,R}_\infty$ so is itself perfectoid. In particular, $\fra^{A,R}_\infty\tensor_R S$ is a perfectoid ideal of $\safty$. 
\end{proof}

\begin{proposition}\label{behavior beta étale morphims}
Let $\rphi\to\spsi$ be as in \autoref{setting étale morphism} and assume further that $\phi$ is surjective. Then, $\upbeta\rphi=\upbeta\spsi\cap R$ and $\upbeta\spsi=\upbeta\rphi S$. More generally, if $\fra$ is any ideal of $S$, then $\upbeta_{\fra\cap R} \rphi=\upbeta_\fra \spsi \cap R$.
Moreover, for all radical ideals $\frab\subset R$, $\upbeta_\frab\rphi S=\sqrt{\upbeta_\frab\rphi S}=\upbeta_{\sqrt{\frab S}}\spsi=\upbeta_{\frab S}\spsi$.
\end{proposition}
\begin{proof}
The inclusion $\upbeta_{\fra\cap R}\rphi \supset \upbeta_\fra\spsi\cap R $ follows directly from \autoref{comm diag extension of pairs} since the contraction of a $\psi$-compatible ideal must be $\phi$-compatible. For ``$\subset$", we proceed by induction. Let $r\in R$.  
By \autoref{lemma extension of perfectoid ideal étale case}, $(r)^{A,R}_\infty \tensor_{R} S$ is the perfectoidization of $(r)$ in $\safty$. If $r\in \upbeta_{(\fra\cap R),1}\rphi$ for an ideal $\fra$ of $S$.
Then, 
\[
    \psi\left(\left(r\right)^{A,R}_\infty \tensor_{R} S\right) =\phi\left( \left(r\right)^{A,R}_\infty\right) \tensor_{R} S \subset (\fra \cap R) \tensor_{R} S \subset \fra.
\]
This implies that $r\in \upbeta_{\fra,1}\spsi$ hence
$\upbeta_{(\fra\cap R),1}\rphi \subset \upbeta_{\fra,1}\spsi$. Suppose that we know that $\upbeta_{(\fra\cap R),i}\rphi \subset \upbeta_{\fra,i}\spsi$ and let $r\in \upbeta_{(\fra\cap R), i+1}\rphi$. Then, 
\[
    \psi\left(\left(r\right)^{A,R}_\infty \tensor_{R} S\right) =\phi\left( \left(r\right)^{A,R}_\infty\right) \tensor_{R} S \subset \upbeta_{(\fra\cap R),i}\rphi \tensor_R S \subset \upbeta_{\fra,i}\spsi.
\]
This implies that $r\in \upbeta_{\fra, i+1}\spsi$ hence
$\upbeta_{(\fra\cap R),i+1}\rphi \subset \upbeta_{\fra,i+1}\spsi$. 
The implication to the last statement follows as in \cite[Definition-Proposition 5.5]{CarvajalFayolleTameRamificationCFPS} but we write it here for the sake of completeness. The first and third equalities follow from the fact that an extension of a radical ideal under an étale morphism is radical. We therefore only need to show the middle equality.
Let $\frab\subset R$ be a radical ideal. We have 
\[ 
\upbeta_{\sqrt{\frab S}}\spsi=\bigcap_{\frab\subset \q\cap R}\upbeta_\q\spsi
\]
and 
\[
\sqrt{\upbeta_\frab\rphi S}=\bigcap_{\q\in \mathcal{C}} \q
\]
where $\mathcal{C}$ is the set of primes $\q\in \Spec S$ with $\upbeta_{\frab\rphi}\subset \q\cap R$. Therefore, to show that $\sqrt{\upbeta_\frab\rphi S}\subset \upbeta_{\sqrt{\frab S}}\rphi$, it suffices to show that $\frab\subset \q\cap R$ for some $\q\in \Spec S$ implies $\upbeta_{\frab}\rphi\subset \upbeta_\q\spsi \cap R$. But $\frab\subset \q\cap R$ implies \[\upbeta_\frab\rphi\subset\upbeta_{\q\cap R}\rphi=\upbeta_{\q}\cap R.\] which is what we wanted. For the other containment, it suffices to show this in case $\frab=\p$ is a prime since the $\upbeta$ construction commutes with taking intersections and we are working in a flat extension. Then, let $\q\in \Spec S$ such that $\upbeta_\p\rphi\subset \q\cap R$. We want to show that $\upbeta_{\sqrt{\pS}}\spsi \subset \q$. By going-down, there is $\q'\subset \q$ lying over $\upbeta_\p$ and by going up there is $\q''\supset \q'$ lying over $\p$ so that $\sqrt{\ps}\subset \q''$. We have 
\begin{align*}
\upbeta_{\q''}\spsi\cap R=\upbeta_{\q''\cap R}\rphi=\upbeta_\p\rphi=\upbeta_{\upbeta_\p\rphi}\rphi&=\upbeta_{\q'\cap R}\rphi\\
&=\upbeta_{\q'}\spsi\cap R.
\end{align*}
Since $\q'\subset \q''$, $\upbeta_{\q'}\spsi\subset\upbeta_{\q''}\spsi$ so by incomparability these must be equal.
Using now that $\sqrt{\pS}\subset \q''$, we get $\upbeta_{\sqrt{\pS}}\spsi\subset \q'\subset \q$, which is what we wanted. 
\end{proof}

\begin{corollary}\label{test ideal of pairs under étale morphisms}
Let $\rphi\to\spsi$ be as in \autoref{setting étale morphism} and assume further that they are domains. Let $\uptau\rphi, \uptau\spsi$ be the respective nonzero smallest compatible ideals. If $\phi$ is surjective, then $\uptau\rphi=\uptau\spsi\cap R$ and $\uptau\rphi S=\uptau\spsi$.
\end{corollary}
\begin{proof}
$\uptau\spsi\cap R$ is compatible so it contains $\uptau\rphi$. Extending to $\spsi$ gives 
\[
\uptau\rphi S\subset (\uptau\spsi\cap R) S \subset \uptau\spsi
\]
but since these are all compatible ideals, we must have 
\[
\uptau\rphi S\supset \uptau\spsi
\]
so this is an equality. 

To show that $\uptau\rphi=\uptau\spsi\cap R$, note that $\upbeta_R\inv(0_R)\subset \upbeta_S\inv(0_S)\cap R$. Indeed, let $\q\in \Spec S$ be such that $\q\cap R= \p $ with $\upbeta_\p\rphi=0$. By \autoref{behavior beta étale morphims}, 
$\upbeta_\q\spsi\cap R=\upbeta_{\q\cap R}\rphi=0$ so $\upbeta_\q\spsi=0$.
Then, by \autoref{beta is locally closed}, $V(\uptau\rphi)\subset V(\uptau\spsi)\cap R$ which implies $\uptau\spsi\cap R \subset \uptau\rphi$. Since the other inclusion is automatic, $\uptau\spsi\cap R = \uptau\rphi$. 
\end{proof}

As usual, these also hold for uniformly perfectoid compatible ideals. 
\begin{proposition}
Let $R\to S$ be as in \autoref{setting étale morphism} and assume $R$ (equivalently $S$) is perfectoid pure. Let $\frab\subset S$ be a radical ideal. Then, $\upbeta_\frab(S)\cap R=\upbeta_{\frab\cap R}(R)$. Moreover, if $\fra\subset R$ is a radical ideal then, $\upbeta_\fra(R)S=\upbeta_{\fra S}(S)$. 
\end{proposition}
\begin{proof}
Take any unramified regular local ring $A$ such that $R$ (and therefore $S$) is a module finite $A$-algebra. Let $\frab\subset S$ be a radical ideal and let $\fra\coloneqq \frab\cap R$. Then, $\upbeta_\frab(S) \cap R$ is a uniformly perfectoid compatible ideal of $R$: if $\phi\in \Hom_R(\rafty,R)$ then $\phi\tensor_R S\in \Hom_S(\safty,S)$ and 
\[
\phi\left(\left(\upbeta_\frab(S)\cap R\right)^A_\infty\right)\subset \left(\phi\left(\left(\upbeta_\frab(S)\cap R\right)^A_\infty\right)\tensor_R S\right)\cap R \subset \upbeta_\frab(S)\cap R
\]
This shows $\upbeta_\frab(S)\cap R\subset \upbeta_\fra(R)$. For the other inclusion, we proceed by induction. Let $r\in \upbeta_{\fra,1}(R)$. Then, for any $x\in (r)^{A, R}_\infty$ the composition
\[
\Hom_R\left(\rafty,R\right)\xlongrightarrow{\Hom_R(\times x, R)}\Hom_R\left(\rafty,R\right) \longrightarrow \Hom_R(R,R)\cong R \to R/\fra
    \]
is zero. Tensoring with $S$ and using the fact that $\fra S\subset \frab$, we get that the composition 
\[
\Hom_S\left(\safty,S\right)\xlongrightarrow{\Hom_S(\times x, S)}\Hom_S\left(\safty,S\right) \longrightarrow \Hom_S(S,S)\cong S \to S/\frab
    \]
is zero. Precomposing with multiplication by any element of $\safty$ on $\safty$ would keep the composition $0$ so for any $\psi\in \Hom_S(\safty,S)$, $\psi((rS)^A_\infty)\subset \frab$ hence $r\in \upbeta_{\frab,1} (S) \cap R$. Now, suppose that we know that $\upbeta_{\fra,i}(R)\subset \upbeta_{\frab,i} (S)\cap R$ and let $r\in \upbeta_{\fra,i+1}(R)$ and $x \in (r)^{A,R}_\infty$.
The composition
\[
\Hom_R\left(\rafty,R\right)\xlongrightarrow{\Hom_R(\times x, R)}\Hom_R\left(\rafty,R\right) \longrightarrow \Hom_R(R,R)\cong R \to R/\upbeta_{\fra,i}(R)
    \]
is zero so tensoring with $S$ and using the fact that $\upbeta_{\fra,i}(R) S\subset \upbeta_{\frab,i}(S)$, we get that the composition 
\[
\Hom_S\left(\safty,S\right)\xlongrightarrow{\Hom_S(\times x, S)}\Hom_S\left(\safty,S\right) \longrightarrow \Hom_S(S,S)\cong S \to S/\upbeta_{\frab,i}(S)
\]
is zero. Precomposing with multiplication by any element of $\safty$ on $\safty$ would keep the composition $0$. Then, for any $\psi\in \Hom_S(\safty,S)$, $\psi((rS)^A_\infty)\subset \upbeta_{\frab,i}(S)$ hence $r\in \upbeta_{\frab,i+1} (S) \cap R$. This shows $\upbeta_\fra(R)=\upbeta_\frab(S)\cap R$.
The implication about $\upbeta$ under extensions follows as in \autoref{behavior beta étale morphims}.
\end{proof}

\begin{corollary}
Let $R\to S$ be as in \autoref{setting étale morphism} and assume further that they are domains. Let $\uptau(R)$ and $\uptau(S)$ be the respective nonzero smallest compatible ideals. If $R$ (equivalently $S$) is perfectoid pure, then $\uptau(R)=\uptau(S)\cap R$ and $\uptau(R) S=\uptau(S)$.
\end{corollary}
\begin{proof}
    Same as \autoref{test ideal of pairs under étale morphisms}.
\end{proof}

\bibliographystyle{skalpha}
\bibliography{MainBib}

\end{document}